\documentclass[11pt,twoside]{amsart}
\usepackage{graphicx}
\usepackage{color}
\usepackage{amssymb}
\usepackage{amsmath}
\usepackage{amsthm}
\usepackage{amsfonts}
\usepackage{amsmath, amsfonts, amssymb}
\newtheorem{theorem}{Theorem}[section]

\newtheorem{proposition}[theorem]{Proposition}

\newtheorem{remark}[theorem]{Remark}

\numberwithin{equation}{section}

\begin{document} 
	\begin{center}\small{In the name of Allah, Most Gracious, Most Merciful.}\end{center}
	\vspace{0.5cm}
	
	\begin{center}\textbf{Classification of Two-Dimensional Algebras and the Number of Non-isomorphic Two-Dimensional Algebras over a Finite Field }\\ 
		\medskip 
		\textbf{Bekbaev U.}\\ 
		\textbf{$^1$ Turin Polytechnic University in Tashkent, Tashkent, Uzbekistan;}\\ \textbf{ uralbekbaev@gmail.com}\\
		
	\end{center} 
	
	Abstract. An updated classification of two-dimensional algebras over an arbitrary base field is provided, along with the number of non-isomorphic such algebras over a finite field. 
	
	\section{Introduction} Classification, up to isomorphism, of all two-dimensional algebras over a given field $\mathbb{F}$ is an interesting problem in algebra. A coordinate-free solution to this problem is offered in \cite{P1}, while coordinate-based solutions are provided in \cite{ABR1} and \cite{B}. However, in \cite{ABR1}, it is solved under the assumption on $\mathbb{F}$ that every second- and third-degree polynomial over $\mathbb{F}$ has a root in it. Of course, there is much other research dealing with this problem in specific cases—for example, commutative algebras, algebras without zero divisors, algebraically closed base fields, and so on \cite{Al1, Al2, An, Bu, C, K}.
	
	An advantage of classifying all two-dimensional algebras over an arbitrary base field is that the classification of any subclass of such algebras can be easily derived from the general classification. For example, using this approach, $30$ classes of algebras are classified in \cite{ABR2} over any base field where every second- and third-degree polynomial has a root.
	
	In \cite{P2}, using the coordinate-free classification \cite{P1}, the number of non-isomorphic two-dimensional algebras over a finite field $\mathbb{F}_q$ is counted. The result is presented as
	\begin{center}
		$q^4+q^3+4q^2+4q+7$, if $\operatorname{char}(\mathbb{F}_q)\neq 2,3$,\\
		$q^4+q^3+4q^2+3q+6$, if $\operatorname{char}(\mathbb{F}_q)= 2$,\\
		$q^4+q^3+4q^2+4q+6$, if $\operatorname{char}(\mathbb{F}_q)= 3$.
	\end{center}
	
	In \cite{V}, one can find another approach to determining the number of such non-isomorphic algebras, though without providing an exact formula in terms of $q$.
	
	The original aim of this paper was to count the number of two-dimensional algebras over $\mathbb{F}_q$ relying on the classification result in \cite{B}. However, the discrepancy between the numbers above and those obtained from \cite{B} forced us to recheck that classification. In \cite{B}, all two-dimensional algebras, except the trivial one, were presented as a disjoint union of five invariant subsets of algebras, and then each class was classified up to isomorphism. Our rechecking revealed some defects in the classification of the fifth subset—namely, the algebras with zero traces. In this paper, we provide an updated version of that classification result over any base field and compute the number of non-isomorphic two-dimensional algebras in the finite field case.

	\section{On the Classification of Two-Dimensional Algebras over an Arbitrary Base Field}
	
	Let $\mathbb{F}$ be a field. Recall that a two-dimensional algebra over $\mathbb{F}$ given by the matrix
	\[A=\left(
	\begin{array}{cccc}
		\alpha _1 & \alpha _2 & \alpha _3 & \alpha _4 \\
		\beta _1 & \beta _2 & \beta _3 & \beta _4
	\end{array}
	\right)\] means the algebra with multiplication
	\[e_1e_1=\alpha _1e_1+\beta _1e_2, \ e_1e_2=\alpha _2e_1+\beta _2e_2,\ e_2e_1=\alpha _3e_1+\beta _3e_2, \ e_2e_2=\alpha _4e_1+\beta _4e_2,\] where $e_1,e_2$ is a basis, $\alpha _1, \alpha _2, \alpha _3,  \alpha _4,
	\beta _1, \beta _2, \beta _3, \beta _4\in \mathbb{F}$. The vectors $$Tr_1(A)=(\alpha _1+\beta _3,\alpha _2+\beta _4), Tr_2(A)=(\alpha _1+\beta _2,\alpha _3+\beta _4)$$ we call the traces of $A$.
	
	The following is the updated variant of  classification result presented in \cite{B}.
	
	\begin{theorem} Any non-trivial two-dimensional algebra over a field $\mathbb{F}$ with $\operatorname{char}(\mathbb{F})\neq 2,3$
		is isomorphic to exactly one of the following algebras, given by their matrices of structure constants. \begin{itemize}
			\item
			$ A_{1}(\mathbf{c})=\left(
			\begin{array}{cccc}
				\alpha_1 & \alpha_2 &1+\alpha_2 & \alpha_4 \\
				\beta_1 & -\alpha_1 & 1-\alpha_1 & -\alpha_2
			\end{array}\right),\ \mbox{where}\ \mathbf{c}=(\alpha_1, \alpha_2, \alpha_4, \beta_1) \in \mathbb{F}^4,$\\ 
			\item	$ A_{2}(\mathbf{c})=\left(
			\begin{array}{cccc}
				\alpha_1 & 0 & 0 & \alpha_4 \\
				1& \beta_2 & 1-\alpha_1&0
			\end{array}\right),\ \mbox{where}\ \mathbf{c}=(\alpha_1,\alpha_4, \beta_2)\in \mathbb{F}^3,\ \alpha_4\neq 0,$ 
			\item	$ A_{3}(\mathbf{c})=\left(
			\begin{array}{cccc}
				\alpha_1 & 0 & 0 & \alpha_4 \\
				0& \beta_2 & 1-\alpha_1&0
			\end{array}\right)\simeq \left(
			\begin{array}{cccc}
				\alpha_1 & 0 & 0 & a^2\alpha_4 \\
				0& \beta_2 & 1-\alpha_1&0
			\end{array}\right),$ where $\mathbf{c}=(\alpha_1,\alpha_4, \beta_2)\in \mathbb{F}^3$ and $0\neq a\in \mathbb{F}$, \item	$ A_{4}(\mathbf{c})=\left(
			\begin{array}{cccc}
				0 & 1 & 1 & 0 \\
				\beta _1& \beta_2 & 1&-1
			\end{array}\right),\ \mbox{where}\ \mathbf{c}=(\beta_1, \beta_2)\in \mathbb{F}^2,$ 
			\item	$A_{5}(\mathbf{c})=\left(
			\begin{array}{cccc}
				\alpha _1 & 0 & 0 & 0 \\
				1 & 2\alpha _1-1 & 1-\alpha _1&0
			\end{array}\right),\ \mbox{where}\ \mathbf{c}=\alpha_1\in \mathbb{F},$
			\item	$ A_{6}(\mathbf{c})=\left(
			\begin{array}{cccc}
				\alpha_1 & 0 & 0 & \alpha_4 \\
				1& 1-\alpha_1 & -\alpha_1&0
			\end{array}\right),\ \mbox{where}\ \mathbf{c}=(\alpha_1,\alpha_4)\in \mathbb{F}^2,\ \alpha_4\neq 0,$ 
			\item	$ A_{7}(\mathbf{c})=\left(
			\begin{array}{cccc}
				\alpha_1 & 0 & 0 & \alpha_4 \\
				0&1-\alpha_1 & -\alpha_1&0
			\end{array}\right)\simeq \left(
			\begin{array}{cccc}
				\alpha_1 & 0 & 0 & a^2\alpha_4 \\
				0& 1-\alpha_1 & -\alpha_1&0
			\end{array}\right),$ where $\mathbf{c}=(\alpha_1,\alpha_4)\in \mathbb{F}^2$ and $ 0\neq a\in \mathbb{F}$,
			\item	$ A_{8}(\mathbf{c})=\left(
			\begin{array}{cccc}
				0 & 1 & 1 & 0 \\
				\beta _1& 1 & 0&-1
			\end{array}\right),\ \mbox{where}\ \mathbf{c}=\beta_1\in\mathbb{F},$
			\item	$ A_{9}(\mathbf{c})=\left(
			\begin{array}{cccc}
				\frac{1}{3} & 0 & 0 & 0 \\
				1 &\frac{2}{3} &-\frac{1}{3}&0
			\end{array}\right),$
			\item	$A_{10}(\mathbf{c})=\left(
			\begin{array}{cccc}
				0 &1 & 1 &1 \\
				\beta_1 &0 &0 &-1
			\end{array}
			\right)\simeq \left(
			\begin{array}{cccc}
				0 &1 & 1 &1 \\
				f(\beta_1,a) &0 &0 &-1
			\end{array}\right),$ where\\ $(\beta _1a^3-3a-1)(\beta _1a^2+\beta_1a+1)(\beta_1^2a^3+6\beta_1a^2+3\beta_1a+\beta_1-2)\neq 0,\\  f(\beta_1,t)=\frac{(\beta_1^2t^3+6\beta_1t^2+3\beta_1t+\beta_1-2)^2}{(\beta_1t^2+\beta_1t+1)^3}$ and  $f(\beta_1,\frac{-1}{2})$ is considered to be $4-\beta_1$.\\
			\item	$ 
			A_{11}(\mathbf{c})=\left(
			\begin{array}{cccc}
				0 &0 & 0 &1 \\
				\beta_1 &0 &0 &0
			\end{array}
			\right)\simeq\left(
			\begin{array}{cccc}
				0 &0 & 0 &1 \\
				\beta_1^2 &0 &0 &0
			\end{array}
			\right)\simeq \left(
			\begin{array}{cccc}
				0 &0 & 0 &1 \\
				a^3\beta_1 &0 &0 &0
			\end{array}\right),\ where\\ \ \mathbf{c}=\beta_1\ and\ 0\neq a\in \mathbb{F}.$\\ Algebras $A_{10}(\beta_1)$ and $A_{11}(\beta'_1)$ are isomorphic if and only if there exists $t\in \mathbb{F}^*$ such that $\beta_1=\beta'_1t^3+2+1/(\beta'_1t^3)$ and $(\beta'_1)^2t^6\neq 1$.\\
			\item	$A_{12}(\mathbf{c})=\left(
			\begin{array}{cccc}
				0 & 1 & 1 &0  \\
				\beta_1 &0& 0 &-1
			\end{array}
			\right)\simeq \left(
			\begin{array}{cccc}
				0 & 1 & 1 & 0 \\
				a^2\beta_1 &0& 0 &-1
			\end{array}
			\right),\ \mbox{where}\ \mathbf{c}=\beta_1\in \mathbb{F},\ 0\neq a\in \mathbb{F}.$\\ Algebras $A_{10}(\beta_1)$ and $A_{12}(\beta'_1)$ are isomorphic  if and only if there exists $t\neq \pm1/2, s\neq 0$ such that $\beta_1=2(2t+1)^2(1-t), s^2\beta'_1=1-t^2$.\\ Algebras $A_{11}(\beta_1)$ and $A_{12}(\beta'_1)$ are isomorphic  if and only if there exists $s,t\in \mathbb{F}^*$ such that $\beta_1=-8t^3, s^2\beta'_1=-t^2$.\
		\end{itemize}
		Any non-trivial 2-dimensional algebra over a field $\mathbb{F}$, $\operatorname{char}(\mathbb{F})= 2$, is isomorphic to only one of the following listed, by their matrices of structure constants,  algebras.  \begin{itemize}
			\item
			$ A_{1,2}(\mathbf{c})=\left(
			\begin{array}{cccc}
				\alpha_1 & \alpha_2 &\alpha_2+1 & \alpha_4 \\
				\beta_1 & \alpha_1 & 1+\alpha_1 & \alpha_2
			\end{array}\right),\ \mbox{where}\ \mathbf{c}=(\alpha_1, \alpha_2, \alpha_4, \beta_1) \in \mathbb{F}^4,$\\ 
			\item	$ A_{2,2}(\mathbf{c})=\left(
			\begin{array}{cccc}
				\alpha_1 & 0 & 0 & \alpha_4 \\
				1& \beta_2 & 1+\alpha_1&0
			\end{array}\right),\ \mbox{where}\ \mathbf{c}=(\alpha_1,\alpha_4, \beta_2)\in \mathbb{F}^3,\ \alpha_4\neq 0,$\\
			$A_{2,2}(\alpha_1,0,1)=\left(
			\begin{array}{cccc}
				\alpha_1 & 0 & 0 &0  \\
				1& 1 & 1+\alpha_1&0
			\end{array}\right),\ \mbox{where}\ \alpha_1\in\mathbb{F},  $\\ 
			\item	$ A_{3,2}(\mathbf{c})=\left(
			\begin{array}{cccc}
				\alpha_1 & 0 & 0 & \alpha_4 \\
				0& \beta_2 & 1+\alpha_1&0
			\end{array}\right)\simeq \left(
			\begin{array}{cccc}
				\alpha_1 & 0 & 0 & a^2\alpha_4 \\
				0& \beta_2 & 1+\alpha_1&0
			\end{array}\right),$ where $\mathbf{c}=(\alpha_1,\alpha_4, \beta_2)\in \mathbb{F}^3$ and $0\neq a\in \mathbb{F}$,\\ \item	$ A_{4,2}(\mathbf{c})=\left(
			\begin{array}{cccc}
				\alpha_1 & 1 & 1 & 0 \\
				\beta _1& \beta_2 & 1+\alpha_1&1
			\end{array}\right)\simeq \left(
			\begin{array}{cccc}
				\alpha_1 & 1 & 1 & 0 \\
				\beta _1+(1+\beta_2)a+a^2& \beta_2 & 1+\alpha_1&1
			\end{array}\right),\ \mbox{where}\ \mathbf{c}=(\alpha_1,\beta_1, \beta_2)\in \mathbb{F}^2,$\\ \item	$ A_{5,2}(\mathbf{c})=\left(
			\begin{array}{cccc}
				\alpha_1 & 0 & 0 & \alpha_4 \\
				1&1+\alpha_1 & \alpha_1&0
			\end{array}\right),\ \mbox{where}\ \mathbf{c}=(\alpha_1,\alpha_4)\in \mathbb{F}^2,\ \alpha_4\neq 0,\\
			\ A_{5,2}(1,0)=\left(
			\begin{array}{cccc}
				1 & 0 & 0 & 0 \\
				1&0 & 1&0
			\end{array}\right),$\\ 
			\item	$ A_{6,2}(\mathbf{c})=\left(
			\begin{array}{cccc}
				\alpha_1 & 0 & 0 & \alpha_4 \\
				0&1+\alpha_1 & \alpha_1&0
			\end{array}\right)\simeq \left(
			\begin{array}{cccc}
				\alpha_1 & 0 & 0 & a^2\alpha_4 \\
				0& 1+\alpha_1 & \alpha_1&0
			\end{array}\right),$ where $\mathbf{c}=(\alpha_1,\alpha_4)\in \mathbb{F}^2$ and $ 0\neq a\in \mathbb{F}$,\\
			\item	$ A_{7,2}(\mathbf{c})=\left(
			\begin{array}{cccc}
				\alpha_1 & 1 & 1 & 0 \\
				\beta _1& 1+\alpha_1 & \alpha_1&1
			\end{array}\right)\simeq \left(
			\begin{array}{cccc}
				\alpha_1 & 1 & 1 & 0 \\
				\beta _1+a\alpha_1+a+a^2& 1+\alpha_1 & \alpha_1&1
			\end{array}\right),\\ \mbox{where}\ \mathbf{c}=(\alpha_1,\beta_1)\in\mathbb{F}^2,\ a\in \mathbb{F},$\\
			\item	$ A_{8,2}(\mathbf{c})=\left(
			\begin{array}{cccc}
				0 &1 & 1 &1 \\
				\beta_1 &0 &0 &-1
			\end{array}
			\right)\simeq \left(
			\begin{array}{cccc}
				0 &1 & 1 &1 \\
				f(\beta_1,a) &0 &0 &-1
			\end{array}\right),$\ where\\ $(\beta _1a^3+a+1)(\beta _1a^2+\beta_1a+1)\beta_1\neq 0$,  and   $f(\beta_1,t)=\beta_1^2\frac{(\beta_1t^3+t+1)^2}{(\beta_1t^2+\beta_1t+1)^3},$\\
			\item	$A_{9,2}(\mathbf{c})=\left(
			\begin{array}{cccc}
				0 &0 & 0 &1 \\
				\beta_1 &0 &0 &0
			\end{array}
			\right)\simeq \left(
			\begin{array}{cccc}
				0 &0 & 0 &1 \\
				\beta_1^2 &0 &0 &0
			\end{array}\right)\simeq \left(
			\begin{array}{cccc}
				0 &0 & 0 &1 \\
				a^3\beta_1 &0 &0 &0
			\end{array}\right),$\ where\  $\mathbf{c}=\beta_1\in \mathbb{F}$ and $0\neq a\in \mathbb{F}$.\\
			Algebras $A_{8,2}(\beta_1)$ and $A_{9,2}(\beta'_1)$ are isomorphic if and only if there exists $t\in \mathbb{F}^*$ such that $\beta_1=\beta'_1t^3+1/(\beta'_1t^3)$ and $\beta'_1t^3\neq 1$. 
			\item	$A_{10,2}(\mathbf{c})=\left(
			\begin{array}{cccc}
				1 & 1 & 1 & 0 \\
				\beta_1 &1& 1 &1
			\end{array}
			\right)\simeq \left(
			\begin{array}{cccc}
				1 & 1 & 1 & 0 \\
				\beta_1+a+a^2 &1& 1 &1
			\end{array}
			\right),\ \mbox{where}\ \mathbf{c}=\beta_1\in \mathbb{F},\ a\in \mathbb{F},$\\
			Algebras $A_{8,2}(\beta_1)$ and $A_{10,2}(\beta'_1)$ are isomorphic if and only if there exists $1\neq s\in \mathbb{F}^*$ such that  $\beta_1=s^3+s$, $\beta'_1=(1+s)/s$.\\
			Algebras $A_{9,2}(\beta_1)$ and $A_{10,2}(\beta'_1)$ are isomorphic if and only if there exists $0\neq s,t\in \mathbb{F}$ such that $\beta_1=s^3$, $t^2+st+(\beta'_1+1)s^2=0$. \item	$A_{11,2}=\left(
			\begin{array}{cccc}
				0 & 1 & 1 & 0 \\
				\beta_1 &0& 0 &1
			\end{array}
			\right)\simeq \left(
			\begin{array}{cccc}
				0 & 1 & 1 & 0 \\
				b^2(\beta_1+a^2) &0& 0 &1
			\end{array}
			\right),$ where $0\neq b$, $a,b\in \mathbb{F}$.\\
			Algebras $A_{8,2}(\beta_1)$ and $A_{11,2}(\beta'_1)$ are isomorphic if and only if  $\beta_1=0, \beta'_1=b^2$ for some $b\in \mathbb{F}$.\\
		\end{itemize}\
		Any non-trivial 2-dimensional algebra over a field $\mathbb{F}$, $\operatorname{char}(\mathbb{F})=3$, is isomorphic to only one of the following listed, by their matrices of structure constants,  algebras.
		\begin{itemize}
			\item
			$ A_{1,3}(\mathbf{c})=\left(
			\begin{array}{cccc}
				\alpha_1 & \alpha_2 &\alpha_2+1 & \alpha_4 \\
				\beta_1 & -\alpha_1 & 1-\alpha_1 & -\alpha_2
			\end{array}\right),\ \mbox{where}\ \mathbf{c}=(\alpha_1, \alpha_2, \alpha_4, \beta_1) \in \mathbb{F}^4,$\\ \item	$ A_{2,3}(\mathbf{c})=\left(
			\begin{array}{cccc}
				\alpha_1 & 0 & 0 & \alpha_4 \\
				1& \beta_2 & 1-\alpha_1&0
			\end{array}\right),\ \mbox{where}\ \mathbf{c}=(\alpha_1,\alpha_4, \beta_2)\in \mathbb{F}^3,\ \alpha_4\neq 0,$\\ 
			\item	$ A_{3,3}(\mathbf{c})=\left(
			\begin{array}{cccc}
				\alpha_1 & 0 & 0 & \alpha_4 \\
				0& \beta_2 & 1-\alpha_1&0
			\end{array}\right)\simeq \left(
			\begin{array}{cccc}
				\alpha_1 & 0 & 0 & a^2\alpha_4 \\
				0& \beta_2 & 1-\alpha_1&0
			\end{array}\right),$\ \ where $\mathbf{c}=(\alpha_1,\alpha_4, \beta_2)\in \mathbb{F}^3$ and $0\neq a\in \mathbb{F}$,\\ 	\item	$ A_{4,3}(\mathbf{c})=\left(
			\begin{array}{cccc}
				0 & 1 & 1 & 0 \\
				\beta _1& \beta_2 & 1&-1
			\end{array}\right),\ \mbox{where}\ \mathbf{c}=(\beta_1, \beta_2)\in \mathbb{F}^2,$\ \\ \item	$ A_{5,3}(\mathbf{c})=\left(
			\begin{array}{cccc}
				\alpha _1 & 0 & 0 & 0 \\
				1 & 2\alpha _1-1 & 1-\alpha _1&0
			\end{array}\right),\ \mbox{where}\ \mathbf{c}=\alpha_1\in \mathbb{F},$\\
			\item	$ A_{6,3}(\mathbf{c})=\left(
			\begin{array}{cccc}
				\alpha_1 & 0 & 0 & \alpha_4 \\
				1& 1-\alpha_1 & -\alpha_1&0
			\end{array}\right),\ \mbox{where}\ \mathbf{c}=(\alpha_1,\alpha_4)\in \mathbb{F}^2,\ \alpha_4\neq 0,$\\
			\item	$ A_{7,3}(\mathbf{c})=\left(
			\begin{array}{cccc}
				\alpha_1 & 0 & 0 & \alpha_4 \\
				0&1-\alpha_1 & -\alpha_1&0
			\end{array}\right)\simeq \left(
			\begin{array}{cccc}
				\alpha_1 & 0 & 0 & a^2\alpha_4 \\
				0& 1-\alpha_1 & -\alpha_1&0
			\end{array}\right),$ where $\mathbf{c}=(\alpha_1,\alpha_4)\in \mathbb{F}^2$ and $ 0\neq a\in \mathbb{F}$,\ \\
			\item	$ A_{8,3}(\mathbf{c})=\left(
			\begin{array}{cccc}
				0 & 1 & 1 & 0 \\
				\beta _1& 1 & 0&-1
			\end{array}\right),\ \mbox{where}\ \mathbf{c}=\beta_1\in\mathbb{F},$\\ 
			\item $A_{9,3}(\mathbf{c})=\left(
			\begin{array}{cccc}
				0 &1 & 1 &1 \\
				\beta_1 &0 &0 &-1
			\end{array}
			\right)\simeq \left(
			\begin{array}{cccc}
				0 &1 & 1 &1 \\
				f(\beta_1,a) &0 &0 &-1
			\end{array}\right),$ where $\mathbf{c}=\beta_1\in\mathbb{F},$\\ $1\neq a\in \mathbb{F}$, $(\beta _1a^3-1)(\beta _1a^2+\beta_1a+1)(\beta_1^2a^3+\beta_1+1)\neq 0$,  $f(\beta_1,t)=\frac{(\beta_1^2t^3+\beta_1+1)^2}{(\beta_1t^2+\beta_1t+1)^3}$ and  $f(\beta_1,1)$ is considered to be $2\beta_1+1$.
			
			\item $A_{10,3}(\mathbf{c})=\left(
			\begin{array}{cccc}
				0 &0 & 0 &1 \\
				\beta_1 &0 &0 &0
			\end{array}
			\right)\simeq\left(
			\begin{array}{cccc}
				0 &0 & 0 &1 \\
				\beta_1^2 &0 &0 &0
			\end{array}
			\right)\simeq \left(
			\begin{array}{cccc}
				0 &0 & 0 &1 \\
				a^3\beta_1 &0 &0 &0
			\end{array}\right),$\ where  $\mathbf{c}=\beta_1$ and $0\neq a\in \mathbb{F}$.\\
			Algebras $A_{9,3}(\beta_1)$ and $A_{10,3}(\beta'_1)$ are isomorphic if and only if there exists $t\in \mathbb{F}^*$ such that $\beta_1=\beta'_1t^3+2+1/(\beta'_1t^3)$ and $(\beta'_1)^2t^6\neq 1$. \\
			\item $A_{11,3}(\mathbf{c})=\left(
			\begin{array}{cccc}
				0 &1 & 1 &0 \\
				\beta_1 &0 &0 &-1
			\end{array}
			\right)\simeq \left(
			\begin{array}{cccc}
				0 &1 & 1 &0 \\
				a^2\beta_1 &0 &0 &-1
			\end{array}\right),$\ where  $\mathbf{c}=\beta_1$ and $0\neq a\in \mathbb{F}$.\\
			Algebras $A_{9,3}(\beta_1)$ and $A_{11,3}(\beta'_1)$ are isomorphic  if and only if there exists $\pm 1\neq t\in \mathbb{F}, s\neq 0$ such that $\beta_1=t^3-1, s^2\beta'_1=1-t^2$. \\
			Algebras $A_{10,3}(\beta_1)$ and $A_{11,3}(\beta'_1)$ are isomorphic  if and only if there exists $s,t\in \mathbb{F}^*$ such that $\beta_1=t^3, s^2\beta'_1=-t^2$.
	\end{itemize} \end{theorem}%
		\begin{proof} 	
		A comparison of this result with the corresponding results in [3] shows that some differences appear only for 	
		algebras $A=\left(
		\begin{array}{cccc}
			\alpha _1 & \alpha _2 & \alpha _2 & \alpha _4 \\
			\beta _1 & -\alpha _1 & -\alpha _1 & -\alpha _2
		\end{array}
		\right)$ which is the fifth class of algebras considered in that paper, including the trivial algebra. Therefore in the proof our concern will be only these algebras and we follow the notations used in [3]. Note that even in this case our proof is nearly parallel to proof that of [3]. 
		
		Note that 
		\[A=\left(
		\begin{array}{cccc}
			\alpha _1 & \alpha _2 & \alpha _2 & \alpha _4 \\
			\beta _1 & -\alpha _1 & -\alpha _1 & -\alpha _2
		\end{array}
		\right),\ g^{-1}=\left(\begin{array}{cc} \xi_1& \eta_1\\ \xi_2& \eta_2\end{array}\right),\ \Delta=\xi_1\eta_2-\xi_2\eta_1\ \mbox{and}\]
		\[A'=\left(
		\begin{array}{cccc}
			\alpha' _1 & \alpha' _2 & \alpha' _2 & \alpha' _4 \\
			\beta' _1 & -\alpha' _1 & -\alpha' _1 & -\alpha' _2
		\end{array}
		\right)=gA(g^{-1})^{\otimes 2}\ \mbox{is equivalent to }\]
		\begin{equation}
			\begin{aligned}
				\alpha' _1=&\frac{1}{\Delta }\left(-\beta _1 \eta _1 \xi _1^2+\alpha _1 \eta _2 \xi _1^2+2 \alpha _1 \eta _1 \xi _1 \xi _2+2 \alpha _2 \eta _2 \xi _1 \xi _2+\alpha _2 \eta _1 \xi _2^2+\alpha _4 \eta _2 \xi _2^2\right),\\		
				\alpha' _2=&\frac{-1}{\Delta }\left(\beta _1 \eta _1^2 \xi _1-2 \alpha _1 \eta _1 \eta _2 \xi _1-\alpha _2 \eta _2^2 \xi _1-\alpha _1 \eta _1^2 \xi _2-2 \alpha _2 \eta _1 \eta _2 \xi _2-\alpha _4 \eta _2^2 \xi _2\right),\\		
				\alpha' _4=&\frac{-1}{\Delta }\left(\beta _1 \eta _1^3-3 \alpha _1 \eta _1^2 \eta _2-3 \alpha _2 \eta _1 \eta _2^2-\alpha _4 \eta _2^3\right),\\		
				\beta' _1=&\frac{1}{\Delta }\left(\beta _1 \xi _1^3-3 \alpha _1 \xi _1^2 \xi _2-3 \alpha _2 \xi _1 \xi _2^2-\alpha _4 \xi _2^3\right).
			\end{aligned}
			\label{sys:1}
		\end{equation}
		{\bf The fifth subset in the $\operatorname{char}(\mathbb{F})\neq 2,3$ case.}

		{\bf Case 1: $\alpha _4\neq 0$.} It follows that if  $\xi _1=0$ and $\eta _2=-\frac{\alpha_2}{\alpha_4}\eta _1$, then $\alpha' _1=0$. Therefore it is sufficient to consider the case $\alpha_1=0$.
		
		{\bf Case 1-1: $\alpha _2\neq 0$.} If $\xi_2=\eta_1=0$, then
		$\alpha' _1=0,\	\alpha' _2=\alpha_2 \eta _2,\ \alpha'_4=\alpha_4\frac{\eta_2^2}{\xi_1},\ \beta' _1=\beta _1\frac{\xi_1^2}{\eta_2}$. So one can make $\alpha' _2=1,\ \alpha'_4=1$. Therefore, without loss of generality, one can assume that $\alpha_1=0,\alpha _2=1, \alpha_4=1$ and \ref{sys:1} system becomes
		
		$\alpha' _1=\frac{1}{\Delta }\left(-\beta _1 \eta _1 \xi _1^2+2\eta _2 \xi _1 \xi _2+\eta _1 \xi _2^2+\eta _2 \xi _2^2\right),$
		
		$\alpha' _2=\frac{-1}{\Delta }\left(\beta _1 \eta _1^2 \xi _1-\eta _2^2 \xi _1-2 \eta _1 \eta _2 \xi _2-\eta _2^2 \xi _2\right),$
		
		$\alpha' _4=\frac{-1}{\Delta }\left(\beta _1 \eta _1^3-3\eta _1 \eta _2^2-\eta _2^3\right),$
		
		$\beta' _1=\frac{1}{\Delta }\left(\beta _1 \xi _1^3-3\xi _1 \xi _2^2- \xi _2^3\right).$
		
		In the $\xi_2\eta_1=0$ case this yields the algebra $\left(
		\begin{array}{cccc}
			0 &1 & 1 &1 \\
			\beta_1 &0 &0 &-1
		\end{array}
		\right)\simeq\left(
		\begin{array}{cccc}
			0 &1 & 1 &1 \\
			4-\beta_1 &0 &0 &-1
		\end{array}
		\right)$. Indeed, if $\eta_1=0$, then $\Delta= \xi_1\eta_2$ and
		
		$\alpha' _1=\xi _1\left(2 \xi _2/\xi _1+(\xi _2/\xi _1)^2\right)$ implies $\xi _2=-2\xi _1,$
		
		$\alpha' _2=\eta_2\left(1+\xi _2/\xi _1\right)=-\eta_2\xi _1=1,$
		
		$\alpha' _4=\frac{1}{\xi_1}\left(\eta _2^2\right)=\frac{1}{\xi_1^3}=1.$ So, if $\xi_1=a, \xi_2=-2a$ and $\eta_2=-1/a$, where $a^3=1$, then 
		
		$\beta' _1=\frac{1}{-1}(\beta _1a^3-12a^3+8a^3)=4-\beta _1.$
		
		In the $\xi_2=0$ case no new algebra appears.
		
		In the $\xi_2\eta_1\neq 0$ case $\alpha' _1=0$ is equivalent to $\frac{ \eta _2}{\eta_1}(2\frac {\xi _1}{\xi_2}+1)-\beta _1 (\frac {\xi _1}{\xi_2})^2+
		1=0$. If $2\frac{\xi _1}{\xi_2}+1=0 $, then this yields 
		$\left(
		\begin{array}{cccc}
			0 &1 & 1 &1 \\
			4 &0 &0 &-1
		\end{array}
		\right)\simeq\left(
		\begin{array}{cccc}
			0 &1 & 1 &1 \\
			0 &0 &0 &-1
		\end{array}
		\right)$.
		
		If	$2\frac{\xi _1}{\xi_2}+1\neq 0 $, then $\frac{ \eta _2}{\eta_1}=\frac{\beta _1 (\frac {\xi _1}{\xi_2})^2-1}{2\frac {\xi _1}{\xi_2}+1}=\frac{\beta _1 \xi^2-1}{2\xi+1}$, where $\xi=\frac {\xi _1}{\xi_2}$, and
		\[\alpha' _2=\frac{\eta _1^2\xi _2}{\Delta }\left(-\beta _1 \frac {\xi _1}{\xi_2}+2\frac{ \eta _2}{\eta_1} +(1+\frac {\xi _1}{\xi_2})(\frac{ \eta _2}{\eta_1})^2 \right)=\]\[\frac{\eta _1^2\xi _2}{\Delta }\left(-\beta _1 \xi+2\frac{\beta _1 \xi^2-1}{2\xi+1} +(1+\xi)(\frac{\beta _1 \xi^2-1}{2\alpha _2\xi+1})^2 \right)=\]
		\[\frac{\eta _1^2\xi _2}{\Delta }\frac{\beta _1^2\xi^5+\beta _1^2\xi^4-2\beta _1\xi^3-4\beta _1\xi^2-(3 +\beta_1)\xi-1}{(2\xi+1)^2}=\]
		\[\frac{\eta _1^2\xi _2}{\Delta }\frac{(\beta _1\xi^3-3\xi-1)(\beta _1\xi^2+\beta _1\xi+1)}{(2\xi+1)^2}.\]
		\[\beta'_1=\frac{\xi_2^3}{\Delta}(\beta _1\xi^3-3\xi-1)\]
		\[\alpha'_4=-\frac{\eta_1^3}{\Delta}(\beta _1 -3(\frac{\eta _2}{\eta _1})^2- (\frac{\eta _2}{\eta _1})^3)=\frac{\eta_1^3}{(2\xi+1)^3\Delta}(\beta _1^3\xi^6 +6\beta_1^2\xi^5-20\beta_1\xi^3-15\beta_1\xi^2+6(1-\beta_1)\xi+2-\beta_1)\]
		Note that the equalities \[\Delta=\xi_2\eta_1(\xi\frac{\beta _1 \xi^2-1}{2\xi+1}-1)=\xi_2\eta_1\frac{\beta _1 \xi^3-3\xi-1}{2\xi+1}, \]\[P(t)=\beta _1^3t^6 +6\beta_1^2t^5-20\beta_1t^3-15\beta_1t^2+6(1-\beta_1)t+2-\beta_1=(\beta_1t^3-3t-1)(\beta_1^2t^3+6\beta_1t^2+3\beta_1t+\beta_1-2)\] are valid.
		
		To make $\alpha'_2=\alpha'_4=1$ there should exist $\xi_0\neq -1/2$ such that $P_1(\xi_0)P_2(\xi_0)P_3(\xi_0)\neq 0$, where
		$P_1(t)=\beta _1t^2+\beta _1t+1, P_2(t)=\beta _1t^3-3t-1, P_3(t)=\beta_1^2t^3+6\beta_1t^2+3\beta_1t+\beta_1-2$.
		
		In this case 
		\[\alpha' _2=\frac{\eta _1^2\xi _2}{\Delta }\frac{(\beta _1\xi_0^3-3\xi_0-1)(\beta _1\xi_0^2+\beta _1\xi_0+1)}{(2\xi_0+1)^2}=\eta _1\frac{\beta _1\xi_0^2+\beta _1\xi_0+1}{2\xi_0+1}.\]
		
		Therefore if $\eta _1=\frac{2\xi_0+1}{\beta _1\xi_0^2+\beta _1\xi_0+1}$, then $\alpha' _2=1$.
		
		Equality  $\alpha'_4=\frac{\eta_1^3}{(2\xi_0+1)^3\Delta}P(\xi_0)=1$ is equivalent to   
		$\xi_2=\frac{\beta_1^2\xi_0^3+6\beta_1\xi_0^2+3\beta_1\xi_0+\beta_1-2}{(\beta_1\xi_0^2+\beta_1\xi_0+1)^2}$.
		
		In this case 
		\[\beta' _1=\frac{\xi_2^3}{\Delta }(\beta _1 \xi_0^3-3\xi_0 -1)=\frac{(2\xi_0+1)\xi_2^2}{\eta_1}=\frac{(\beta_1^2\xi_0^3+6\beta_1\xi_0^2+3\beta_1\xi_0+\beta_1-2)^2}{(\beta_1\xi_0^2+\beta_1\xi_0+1)^3}\] and we obtain MSC
		\[\left(\begin{array}{cccc}
			0 &1 & 1 &1 \\
			4-\beta_1 &0 &0 &-1
		\end{array}
		\right)\simeq A_{10}(\mathbf{c})=\left(
		\begin{array}{cccc}
			0 &1 & 1 &1 \\
			\beta_1 &0 &0 &-1
		\end{array}
		\right)\simeq \left(
		\begin{array}{cccc}
			0 &1 & 1 &1 \\
			\beta'_1(a) &0 &0 &-1
		\end{array}\right),\] where  $\frac{-1}{2}\neq a\in \mathbb{F}$, $(\beta _1a^3-3a-1)(\beta _1a^2+\beta_1a+1)(\beta_1^2a^3+6\beta_1a^2+3\beta_1a+\beta_1-2)\neq 0$,  and   \[\beta' _1(t)=\frac{(\beta_1^2t^3+6\beta_1t^2+3\beta_1t+\beta_1-2)^2}{(\beta_1t^2+\beta_1t+1)^3}.\] It follows that $\beta' _1(\frac{-1}{2})=4-\beta_1$.
		
		{\bf Case 1-2: $\alpha _2= 0$.} The system becomes\\
		$\begin{array}{ll} \alpha' _1=\frac{1}{\Delta }\left(-\beta _1 \eta _1 \xi _1^2+\alpha _4 \eta _2 \xi _2^2\right),&
			\alpha' _2=\frac{-1}{\Delta }\left(\beta _1 \eta _1^2 \xi _1-\alpha _4 \eta _2^2 \xi _2\right),\\
			\alpha' _4=\frac{-1}{\Delta }\left(\beta _1 \eta _1^3-\alpha _4 \eta _2^3\right),&\beta' _1=\frac{1}{\Delta }\left(\beta _1 \xi _1^3-\alpha _4 \xi _2^3\right)\end{array}.$ \\ To make $\alpha' _1=\alpha' _2=0$ we must have
		$\xi_1=\eta_2=0$ or $\xi_2=\eta_1=0$. If $\xi_2=\eta_1=0$ and $\xi_1=\alpha _4 \eta _2^2$, then $\alpha' _4=\frac{-1}{\Delta }\left(\beta _1 \eta _1^3-\alpha _4 \eta _2^3\right)=\frac{-1}{\xi_1}(-\alpha _4 \eta _2^2)=1$ and $\beta' _1=\frac{1}{\Delta }\left(\beta _1 \xi _1^3-\alpha _4 \xi _2^3\right)=\frac{\xi _1^2}{\eta _2}\beta _1= \beta _1\alpha _4^2\eta _2^3$, so it yields the algebra
		\[A_{11}(\mathbf{c})=\left(
		\begin{array}{cccc}
			0 &0 & 0 &1 \\
			\beta_1 &0 &0 &0
		\end{array}
		\right)\simeq \left(
		\begin{array}{cccc}
			0 &0 & 0 &1 \\
			a^3\beta_1 &0 &0 &0
		\end{array}\right),\] where  $\mathbf{c}=\beta_1\in \mathbb{F}$ and $0\neq a\in \mathbb{F}$. 
		
		{\bf Case 2: $\alpha _4= 0$.} The system \ref{sys:1} becomes
		
		$\alpha' _1=\frac{1}{\Delta }\left(-\beta _1 \eta _1 \xi _1^2+\alpha _1 \eta _2 \xi _1^2+2 \alpha _1 \eta _1 \xi _1 \xi _2+2 \alpha _2 \eta _2 \xi _1 \xi _2+\alpha _2 \eta _1 \xi _2^2\right),$
		
		$\alpha' _2=\frac{-1}{\Delta }\left(\beta _1 \eta _1^2 \xi _1-2 \alpha _1 \eta _1 \eta _2 \xi _1-\alpha _2 \eta _2^2 \xi _1-\alpha _1 \eta _1^2 \xi _2-2 \alpha _2 \eta _1 \eta _2 \xi _2\right),$
		
		$\alpha' _4=\frac{-\eta _1}{\Delta }\left(\beta _1 \eta _1^2-3 \alpha _1 \eta _1 \eta _2-3 \alpha _2 \eta _2^2\right),$
		
		$\beta' _1=\frac{\xi _1}{\Delta }\left(\beta _1 \xi _1^2-3 \alpha _1 \xi _1 \xi _2-3 \alpha _2\xi _2^2\right).$
		
		If $\eta _1 = 0$, then $\Delta=\xi_1\eta_2$, $\alpha'_4=0$ and 
		
		$\alpha' _1=\xi _1\left(\alpha _1+2\alpha _2\frac{\xi_2}{\xi_1}\right),$\
		
		$\alpha' _2=\alpha _2\eta_2,$\
		
		$\beta' _1=\frac{ \xi _1^2}{\eta_2 }\left(\beta _1-3 \alpha _1 \frac{\xi_2}{\xi_1}-3 \alpha _2(\frac{\xi_2}{\xi_1})^2\right).$
		
		{\bf Case 2-1: $\alpha_2\neq 0$.} If $\frac{\xi_2}{\xi_1}=\frac{-\alpha _1}{2\alpha _2}$, then  $\alpha' _1=0$ $\alpha' _2\neq 0$. After $\alpha' _2$ can be made $1$ to get\\
		$A_{12}(\mathbf{c})=
		\begin{pmatrix}
			0 & 1 & 1 & 0 \\
			\beta_1 &0& 0 &-1
		\end{pmatrix}
		\simeq \left(
		\begin{array}{cccc}
			0 & 1 & 1 & 0 \\
			a^2\beta_1 &0& 0 &-1
		\end{array}
		\right),$ where $\mathbf{c}=\beta_1\in \mathbb{F}$ and $0\neq a\in \mathbb{F}$.

		{\bf Case 2-2: $\alpha_2= 0$.} The corresponding system is
		
		$\alpha' _1=\frac{\xi _1}{\Delta }\left(-\beta _1 \eta _1 \xi _1+\alpha _1 \eta _2 \xi _1+2 \alpha _1 \eta _1 \xi _2\right),$
		
		$\alpha' _2=\frac{-\eta _1}{\Delta }\left(\beta _1 \eta _1 \xi _1-2 \alpha _1  \eta _2 \xi _1-\alpha _1 \eta _1 \xi _2\right),$
		
		$\alpha' _4=\frac{-\eta _1^2}{\Delta }\left(\beta _1 \eta _1-3 \alpha _1 \eta _2\right),$
		
		$\beta' _1=\frac{\xi _1^2}{\Delta }\left(\beta _1 \xi _1-3 \alpha _1 \xi _2\right).$
		
		In the proses of rechecking the following isomorphisms among the above classes of algebras are found:\\
		$A_{11}(\beta_1)\simeq A_{11}(\beta_1^2),$ as far as $gA_{11}(\beta_1)(g^{-1}\otimes g^{-1})=A_{11}(\beta_1^2)$, where  $g=\begin{pmatrix}0 & 1/\beta_1\\
			1&0\end{pmatrix}, \beta_1\neq 0$.\\
		Algebras $A_{10}(\beta_1)$ and $A_{11}(\beta'_1)$ are isomorphic if and only if there exists $t\in \mathbb{F}^*$ such that $\beta_1=\beta'_1t^3+2+1/(\beta'_1t^3)$ and $(\beta'_1)^2t^6\neq 1$. In this case $gA_{10}(\beta_1)(g^{-1}\otimes g^{-1})=A_{11}(\beta'_1)$, where  $g=\begin{pmatrix}t^2+1/(\beta'_1t) & 1/(\beta'_1t)\\
			t+1/(\beta'_1t^2)&t\end{pmatrix}$.\\
		
			Algebras $A_{10}(\beta_1)$ and $A_{12}(\beta'_1)$ are isomorphic  if and only if there exists $t\neq \pm1/2, s\neq 0$ such that $\beta_1=2(2t+1)^2(1-t), s^2\beta'_1=1-t^2$. In this case $gA_{10}(\beta_1)(g^{-1}\otimes g^{-1})=A_{12}(\beta'_1)$, where  $g=\begin{pmatrix}
			(2t+1)s & s\\
			(2t+1)(1-t)&t\end{pmatrix}$.\\
		Algebras $A_{11}(\beta_1)$ and $A_{12}(\beta'_1)$ are isomorphic  if and only if there exists $s,t\in \mathbb{F}^*$ such that $\beta_1=-8t^3, s^2\beta'_1=-t^2$. In this case $gA_{11}(\beta_1)(g^{-1}\otimes g^{-1})=A_{12}(\beta'_1)$, where  $g=\begin{pmatrix}
			2ts & s\\
			-2t^2&t\end{pmatrix}$.\\
		Now we consider case $\operatorname{char}(\mathbb{F})=2$. The approach is the same.
		
		{\bf The fifth subset in $\operatorname{char}(\mathbb{F})=2$ case.}	
		
		{\bf Case 1: $\alpha _4\neq 0$.} If  $\xi _1=0$ and $\eta _2=-\frac{\alpha_2}{\alpha_4}$, then $\alpha' _1=0$. So further we consider only $\alpha_1=0$ case. 
		
		{\bf Case 1.1:  $\alpha _2\neq  0$.}  If $\xi_2=\eta_1=0$, then
		$\alpha' _1=0,\	\alpha' _2=\alpha_2 \eta _2,\ \alpha'_4=\alpha_4\frac{\eta_2^2}{\xi_1},\ \beta' _1=\beta _1\frac{\xi_1^2}{\eta_2}$. Therefore, without loss of generality, one can assume that $\alpha_1=0,\	\alpha _2=1,\ \alpha_4=1$. The corresponding system is:
		
		$\alpha' _1=\frac{1}{\Delta }\left(-\beta _1 \eta _1 \xi _1^2+\eta _1 \xi _2^2+\eta _2 \xi _2^2\right),$
		
		$\alpha' _2=\frac{-1}{\Delta }\left(\beta _1 \eta _1^2 \xi _1-\eta _2^2 \xi _1-\eta _2^2 \xi _2\right),$
		
		$\alpha' _4=\frac{-1}{\Delta }\left(\beta _1 \eta _1^3-\eta _1 \eta _2^2-\eta _2^3\right),$
		
		$\beta' _1=\frac{1}{\Delta }\left(\beta _1 \xi _1^3-\xi _1 \xi _2^2- \xi _2^3\right).$
		
		If $\xi_2=\eta_1=0$, then $\Delta= \xi_1\eta_2$ and		
		$\alpha' _1=0,$
		$\alpha' _2=\eta_2\xi _1,$
		$\alpha' _4=\frac{\eta _2^2}{\xi_1}$,  
		$\beta' _1=\frac{\xi _1^2}{\eta _2}\beta_1.$ By $\eta_2=1,\xi _1=1$ we obtain  $\left(
		\begin{array}{cccc}
			0 &1 & 1 &1 \\
			\beta_1 &0 &0 &-1
		\end{array}
		\right)\simeq\left(
		\begin{array}{cccc}
			0 &1 & 1 &1 \\
			\beta_1 &0 &0 &-1
		\end{array}
		\right)$.
		
		In the $\xi_2\eta_1\neq 0$ case $\alpha' _1=0$ if and only if $\frac{ \eta _2}{\eta_1}=\beta _1 \xi^2-1$, where $\xi=\frac {\xi _1}{\xi_2}$. In this case 
		\[\alpha' _2=\frac{\eta _1^2\xi _2}{\Delta }\left(-\beta _1 \frac {\xi _1}{\xi_2}+(1+\frac {\xi _1}{\xi_2})(\frac{ \eta _2}{\eta_1})^2 \right)=\]\[\frac{\eta _1^2\xi _2}{\Delta }\left(-\beta _1 \xi +(1+\xi)(\beta _1 \xi^2-1)^2 \right)=\]
		\[\frac{\eta _1^2\xi _2}{\Delta }(\beta _1^2\xi^5+\beta _1^2\xi^4-(3 +\beta_1)\xi-1)=\]
		\[\frac{\eta _1^2\xi _2}{\Delta }(\beta _1\xi^3-\xi-1)(\beta _1\xi^2+\beta _1\xi+1).\]
		\[\beta'_1=\frac{\xi_2^3}{\Delta}(\beta _1\xi^3-\xi-1)\]
		\[\alpha'_4=-\frac{\eta_1^3}{\Delta}(\beta _1 -(\frac{\eta _2}{\eta _1})^2- (\frac{\eta _2}{\eta _1})^3)=\frac{\eta_1^3}{\Delta}(\beta _1^3\xi^6 -\beta_1\xi^2-\beta_1)\]
		Note that the equalities \[\Delta=\xi_2\eta_1(\xi(\beta _1 \xi^2-1)-1)=\xi_2\eta_1(\beta _1 \xi^3-\xi-1), \]\[P(t)=\beta _1^3t^6 -\beta_1t^2-\beta_1=(\beta_1t^3+t+1)^2\beta_1\] are valid.
		
		To make $\alpha'_2=\alpha'_4=1$ one should have $\beta_1\neq 0$  and there should exist $\xi_0$ such that $P_1(\xi_0)P_2(\xi_0)\neq 0$, where
		$P_1(t)=\beta _1t^2+\beta _1t+1, P_2(t)=\beta _1t^3+t+1$.
		
		In this case 
		\[\alpha' _2=\frac{\eta _1^2\xi _2}{\Delta }(\beta _1\xi_0^3-\xi_0-1)(\beta _1\xi_0^2+\beta _1\xi_0+1)=\eta _1(\beta _1\xi_0^2+\beta _1\xi_0+1).\]
		
		Therefore, if $\eta _1=\frac{1}{\beta _1\xi_0^2+\beta _1\xi_0+1}$, then $\alpha' _2=1$.
		
		Equality  $\alpha'_4=\frac{\eta_1^3}{(2\xi_0+1)^3\Delta}P(\xi_0)=1$ is equivalent to   
		$\xi_2=\beta_1\frac{\beta_1\xi_0^3+\xi_0+1}{(\beta_1\xi_0^2+\beta_1\xi_0+1)^2}$.
		
		In this case 
		\[\beta' _1=\frac{\xi_2^3}{\Delta }(\beta _1 \xi_0^3-\xi_0 -1)=\beta_1^2\frac{(\beta_1\xi_0^3+\xi_0+1)^2}{(\beta_1\xi_0^2+\beta_1\xi_0+1)^3}\] and this yields MSC
		\[ A_{8.2}(\mathbf{c})=\left(
		\begin{array}{cccc}
			0 &1 & 1 &1 \\
			\beta_1 &0 &0 &-1
		\end{array}
		\right)\simeq \left(
		\begin{array}{cccc}
			0 &1 & 1 &1 \\
			\beta'_1(a) &0 &0 &-1
		\end{array}\right),\] where  $(\beta _1a^3+a+1)(\beta _1a^2+\beta_1a+1)\beta_1\neq 0$  and   \[\beta' _1(t)=\beta_1^2\frac{(\beta_1t^3+t+1)^2}{(\beta_1t^2+\beta_1t+1)^3}.\] 
		
		{\bf Case 1.2: $\alpha _2= 0$.} The system becomes\\
		$\begin{array}{ll} \alpha' _1=\frac{1}{\Delta }\left(-\beta _1 \eta _1 \xi _1^2+\alpha _4 \eta _2 \xi _2^2\right),&
			\alpha' _2=\frac{-1}{\Delta }\left(\beta _1 \eta _1^2 \xi _1-\alpha _4 \eta _2^2 \xi _2\right),\\
			\alpha' _4=\frac{-1}{\Delta }\left(\beta _1 \eta _1^3-\alpha _4 \eta _2^3\right),&\beta' _1=\frac{1}{\Delta }\left(\beta _1 \xi _1^3-\alpha _4 \xi _2^3\right)\end{array}$ \\ If $\xi_2=\eta_1=0$, then $\alpha' _4=\frac{-1}{\Delta }\left(\beta _1 \eta _1^3-\alpha _4 \eta _2^3\right)=\frac{\alpha _4 \eta _2^2}{\xi_1}$ and  $\beta' _1=\frac{1}{\Delta }\left(\beta _1 \xi _1^3-\alpha _4 \xi _2^3\right)=\frac{\xi _1^2}{\eta _2}\beta _1= \beta _1\alpha _4^2\eta _2^3$. Therefore one has
		\[A_{9.2}(\mathbf{c})=\left(
		\begin{array}{cccc}
			0 &0 & 0 &1 \\
			\beta_1 &0 &0 &0
		\end{array}
		\right)\simeq \left(
		\begin{array}{cccc}
			0 &0 & 0 &1 \\
			a^3\beta_1 &0 &0 &0
		\end{array}\right),\] where  $\mathbf{c}=\beta_1\in \mathbb{F}$ and $0\neq a\in \mathbb{F}$. 
		
		{\bf Case 2: $\alpha _4= 0$.} The system becomes
		
		$\alpha' _1=\frac{1}{\Delta }\left(-\beta _1 \eta _1 \xi _1^2+\alpha _1 \eta _2 \xi _1^2+\alpha _2 \eta _1 \xi _2^2\right),$
		
		$\alpha' _2=\frac{-1}{\Delta }\left(\beta _1 \eta _1^2 \xi _1-\alpha _2 \eta _2^2 \xi _1-\alpha _1 \eta _1^2 \xi _2\right),$
		
		$\alpha' _4=\frac{-\eta _1}{\Delta }\left(\beta _1 \eta _1^2- \alpha _1 \eta _1 \eta _2- \alpha _2 \eta _2^2\right),$
		
		$\beta' _1=\frac{\xi _1}{\Delta }\left(\beta _1 \xi _1^2- \alpha _1 \xi _1 \xi _2- \alpha _2\xi _2^2\right).$
		
		If $\eta _1 = 0$, then $\alpha'_4=0$, $\Delta=\xi_1\eta_2$ and
		
		$\alpha' _1=\xi _1\alpha _1,$\
		
		$\alpha' _2=\alpha _2\eta_2,$\
		
		$\beta' _1=\frac{ \xi _1^2}{\eta_2 }\left(\beta _1- \alpha _1 \frac{\xi_2}{\xi_1}- \alpha _2(\frac{\xi_2}{\xi_1})^2\right).$
		
		{\bf Case 2-1: $\alpha_1\neq 0, \alpha_2\neq 0$.} We may assume  $\alpha' _1=1$ $\alpha' _2= 1$ to get\\
		$A_{10,2}(\mathbf{c})=\left(
		\begin{array}{cccc}
			1 & 1 & 1 & 0 \\
			\beta_1 &1& 1 &1
		\end{array}
		\right)\simeq \left(
		\begin{array}{cccc}
			1 & 1 & 1 & 0 \\
			a^2+a+\beta_1 &1& 1 &1
		\end{array}
		\right),$ where $\mathbf{c}=\beta_1\in \mathbb{F}$ and $a\in \mathbb{F}$.

		{\bf Case 2-2: $\alpha_1=0, \alpha_2\neq 0$.} In this we obtain\\
		$A_{11,2}(\mathbf{c})=\left(
		\begin{array}{cccc}
			0 & 1 & 1 & 0 \\
			\beta_1 &0& 0 &1
		\end{array}
		\right)\simeq \left(
		\begin{array}{cccc}
			0 & 1 & 1 & 0 \\
			b^2(\beta_1+a^2) &0&0 &1
		\end{array}
		\right),$ where $\mathbf{c}=\beta_1\in \mathbb{F}$ and $a,b\in \mathbb{F}, b\neq 0$.
		
		{\bf Case 2-3: $\alpha_1\neq 0, \alpha_2= 0$.}	This yields\\
		$A'=\left(
		\begin{array}{cccc}
			1 & 0 & 0 & 0 \\
			0 &-1& -1 &0
		\end{array}
		\right)$, which is isomorphic to $A_{11,2}(0)$.
		
		{\bf Case 2-4: $\alpha_1= 0, \alpha_2= 0$.} In this case one has
		$A=\left(
		\begin{array}{cccc}
			0 & 0 & 0 & 0 \\
			1 &0&0 &0\end{array}\right)\simeq A_{9,2}(0).$
		
		In $\operatorname{char}(\mathbb{F})=2$ case among the above class of algebras there are the following isomorphisms:\\
		Algebras $A_{8,2}(\beta_1)$ and $A_{9,2}(\beta'_1)$ are isomorphic if and only if there exists $t\in \mathbb{F}^*$ such that $\beta_1=\beta'_1t^3+1/(\beta'_1t^3)$ and $(\beta'_1)^2t^6\neq 1$. In this case $gA_{8,2}(\beta_1)(g^{-1}\otimes g^{-1})=A_{9,2}(\beta'_1)$, where  $g=\begin{pmatrix}t^2+1/(\beta'_1t) & 1/(\beta'_1t)\\
			t+1/(\beta'_1t^2)&t\end{pmatrix}$.\\
		Algebras $A_{8,2}(\beta_1)$ and $A_{10,2}(\beta'_1)$ are isomorphic if and only if there exists $0,1\neq s\in \mathbb{F}^*$ such that $\beta_1=s^3+s$ and $\beta'_1=(1+s)/s$. In this case $gA_{8,2}(\beta_1)(g^{-1}\otimes g^{-1})=A_{10,2}(\beta'_1)$, where $g=\begin{pmatrix}s^2+s& s\\
			s^2+s&1\end{pmatrix}$. .\\
		Algebras $A_{8,2}(\beta_1)$ and $A_{11,2}(\beta'_1)$ are isomorphic if and only if  $\beta_1=0,\beta'_1=b^2$ for some $b\in \mathbb{F}$. In this case $gA_{8.2}(0)(g^{-1}\otimes g^{-1})=A_{11,2}(b^2)$, where  $g=\begin{pmatrix}0& 1\\
			b&1\end{pmatrix}$ if $b\neq 0$, otherwise $g=\begin{pmatrix}1& 1\\
			0&1\end{pmatrix}$  .\\
		Algebras $A_{9,2}(\beta_1)$ and $A_{10,2}(\beta'_1)$ are isomorphic if and only if there exists $s\neq 0,t\in \mathbb{F}$ such that $\beta_1=s^3$, $t^2+st+(\beta'_1+1)s^2=0$. In this case $gA_{9,2}(\beta_1)(g^{-1}\otimes g^{-1})=A_{10,2}(\beta'_1)$, where  $g=\begin{pmatrix}s^2&s\\
			s^2+st&t\end{pmatrix}$.\\
		Algebras $A_{9,2}(\beta_1)$ and $A_{11,2}(\beta'_1)$ are not isomorphic.\\		
		Algebras $A_{10,2}(\beta_1)$ and $A_{11,2}(\beta'_1)$ are not isomorphic.\\		
		{\bf The fifth subset in $\operatorname{char}(\mathbb{F})=3$ case.}\\
		The proof is similar to that of $\operatorname{char}(\mathbb{F})\neq 2,3$ except for the following case.\\
		{\bf Case 2-2: $\alpha_2= 0$.} The corresponding system is
		
		$\alpha' _1=\frac{\xi _1}{\Delta }\left(-\beta _1 \eta _1 \xi _1+\alpha _1 \eta _2 \xi _1+2 \alpha _1 \eta _1 \xi _2\right),$
		
		$\alpha' _2=\frac{-\eta _1}{\Delta }\left(\beta _1 \eta _1 \xi _1-2 \alpha _1  \eta _2 \xi _1-\alpha _1 \eta _1 \xi _2\right),$
		
		$\alpha' _4=\frac{-\eta _1^2}{\Delta }\left(\beta _1 \eta _1-3 \alpha _1 \eta _2\right),$
		
		$\beta' _1=\frac{\xi _1^2}{\Delta }\left(\beta _1 \xi _1-3 \alpha _1 \xi _2\right).$
		
		If $\eta_1=0$, then $\alpha'_2=\alpha'_4=0$ and $\alpha'_1=\xi_1\alpha_1, \beta' _1=\frac{\xi^2_1}{\eta_2}\beta _1.$ Therefore, one can have $A'=\left(
		\begin{array}{cccc}
			1 & 0 & 0 & 0 \\
			0 &-1& -1 &0
		\end{array}
		\right)\simeq A_{11,3}(0)$, or $A_{12,3}=\left(
		\begin{array}{cccc}
			1 & 0 & 0 & 0 \\
			1 &-1& -1 &0
		\end{array}
		\right)\simeq A_{9,3}(0)$ or $A=\left(
		\begin{array}{cccc}
			0 & 0 & 0 & 0 \\
			1 &0& 0 &0
		\end{array}
		\right)\simeq A_{10,3}(0)$.
		
		Moreover the following are true.\\ Algebras $A_{9,3}(\beta_1)$ and $A_{10,3}(\beta'_1)$ are isomorphic if and if there exists $t\in F^*$ such that $\beta_1=\beta'_1t^3+2+1/(\beta'_1t^3)$ and $(\beta'_1)^2t^6\neq 1$. In this case $gA_{9,3}(\beta_1)(g^{-1}\otimes g^{-1})=A_{10,3}(\beta'_1)$, where  $g=\begin{pmatrix}t^2+1/(\beta'_1t) & 1/(\beta'_1t)\\
			t+1/(\beta'_1t^2)&t\end{pmatrix}$.\\
		Algebras $A_{9,3}(\beta_1)$ and $A_{11,3}(\beta'_1)$ are isomorphic  if and only if there exists $\pm 1\neq t\in F, s\neq 0$ such that $\beta_1=t^3-1, s^2\beta'_1=1-t^2$. In this case $gA_{9,3}(\beta_1)(g^{-1}\otimes g^{-1})=A_{11,3}(\beta'_1)$, where  $g=\begin{pmatrix}
			(1-t)s & s\\
			t^2+t+1&t\end{pmatrix}$.\\
		Algebras $A_{10,3}(\beta_1)$ and $A_{11,3}(\beta'_1)$ are isomorphic  if and only if there exists $s,t\in F^*$ such that $\beta_1=t^3, s^2\beta'_1=-t^2$. In this case $gA_{10,3}(\beta_1)(g^{-1}\otimes g^{-1})=A_{12,3}(\beta'_1)$, where  $g=\begin{pmatrix}
			-ts & s\\
			t^2&t\end{pmatrix}$.	
	\end{proof}	
	
	\section{The number of none-isomorphic two-dimensional algebras over a finite field}
	
	Let $\mathbb{F}_q$ be a finite field with $q$ elements, where $q = p^k$ for a prime $p$, and $\sigma$ be a fixed generator for the cyclic group $\mathbb{F}_q^*$ . 
	
	In this section we are going to count the number of none-isomorphic two-dimensional algebras over a given finite field $\mathbb{F}_q$. For that the above presented classification theorem can be used directly for the first four classes of algebras, that is for algebras for which at least one of the traces is not zero. 
	
	\begin{theorem} The number of non-isomorphic two-dimensional algebras, for which at least one of the traces is not zero,  is as follows:
		\begin{itemize}	
			\item $\operatorname{char}(\mathbb{F}_q)\neq 2,3$ case:\ \	\
			$q^4+q^3+4q^2+4q+1.$
			\item $\operatorname{char}(\mathbb{F}_q)= 3$ case:\ \ \
			$q^4+q^3+4q^2+4q.$	
			\item $\operatorname{char}(\mathbb{F}_q)= 2$ case.\ \ \	
			$q^4+q^3+4q^2+3q.$
			
		\end{itemize}
	\end{theorem}
	\begin{proof} Let $\mathbb{F}^*=\{1,\sigma,\sigma^2,...,\sigma^{q-2}\}.$
		\begin{itemize}	
			\item The parameters in $A_{1}, A_{4}, A_{5}$ and $A_{8}$ classes are free therefore the equalities $\vert A_{1}\vert=q^4, \vert A_{4}\vert=q^2, \vert A_{5}\vert=q$, $\vert A_{8}\vert=q$ are evident. In $A_{2}, A_{6}$ cases the only constraint is $\alpha_4\neq 0$, therefore $\vert A_{2}\vert=q^2(q-1), \vert A_{6}\vert=q(q-1)$ are valid. In the class $A_3$, the parameter $a^2\alpha_4$ can be reduced to $0$, $1$, or $s$,
			while $\alpha_1$ and $\beta_2$ are free parameters. Therefore
			$\vert A_{3}\vert=3q^2.$ Similarly, $\vert A_{7}\vert=3q$ and $\vert A_{9}\vert=1$. The total number is $q^4+q^3+4q^2+4q+1$.
			
			\item $\vert A_{1,3}\vert=q^4, \vert A_{2,3}\vert=q^2(q-1), \vert A_{3,3}\vert=3q^2$, $\vert A_{4,3}\vert=q^2$, $\vert A_{5,3}\vert=q, \vert A_{6,3}\vert=q(q-1), \vert A_{7,3}\vert=3q$, $\vert A_{8,3}\vert=q$ with total number of $q^4+q^3+4q^2+4q.$
			
			\item Similarly, because of free variables the equalities $\vert A_{1,2}\vert=q^4, \vert A_{2,2}\vert=q^2(q-1)+q)$. In $ A_{3,2}$ case $a^2\alpha_4$ can be maid $0$ or $1$, so $\vert A_{3,2}\vert=2q^2$. Note that $x\to x^2+(1+\beta_2)x+\beta_1$, where $\beta_1,\beta_2$ are fixed, is surjective if $\beta_2=1$, otherwise $x$ and $x+1+\beta_2$ have the same image. Therefore, we have $\vert A_{4,2}\vert=q+2(q-1)q=2q^2-q$. $\vert A_{5,2}\vert=q(q-1)+1$ is clear. In $A_{6,2}$ case $a^2\alpha_4$ can be maid $0$ or $1$, so $\vert A_{6,2}\vert=2q$. In $A_{7,2}$ case the situation is similar to that of $A_{4,2}$ and therefore $\vert A_{7,2}\vert=1+2(q-1)=2q-1$.
		\end{itemize}
	\end{proof}
	Unfortunately we couldn't use the above classification result in counting the number of none-isomorphic algebras for which both traces are zero, mainly because of difficult dynamical behaviour of the rational map $f(s,t)=\frac{(s^2t^3+6st^2+3st+s-2)^2}{(st^2+st+1)^3}$.
	
	Here we are going to offer another way of classification of two-dimensional algebras with zero traces. We use it mainly for finite field case, but it may be useful in other cases as well. 
	
	It can be checked directly that \ref{sys:1} implies the following equalities:
	\begin{equation}
		\begin{aligned}
			(\alpha'_1)^2+\alpha'_2\beta'_1&=(\alpha_1^2+\alpha_2\beta_1)\xi_1^2 +(\alpha_1\alpha_2+\alpha_4\beta_1)\xi_1\xi_2+ (\alpha_2^2-\alpha_1\alpha_4)\xi_2^2,\\
			\alpha'_1\alpha'_2+\alpha'_4\beta'_1&=2(\alpha_1^2+\alpha_2\beta_1)\xi_1\eta_1 +(\alpha_1\alpha_2+\alpha_4\beta_1)(\xi_1\eta_2+\xi_2\eta_1)+ 2(\alpha_2^2-\alpha_1\alpha_4)\xi_2\eta_2,\\
			(\alpha'_2)^2-\alpha'_1\alpha'_4&=(\alpha_1^2+\alpha_2\beta_1)\eta_1^2 +(\alpha_1\alpha_2+\alpha_4\beta_1)\eta_1\eta_2+ (\alpha_2^2-\alpha_1\alpha_4)\eta_2^2.
		\end{aligned}
	\end{equation}
	
	Therefore, first we find a representative $C=(x_0,y_0,z_0)$ from each orbit of the following action  
	\\
	$x'=x\xi_1^2 +y\xi_1\xi_2+ z\xi_2^2,$\\
	$y'=2x\xi_1\eta_1 +y(\xi_1\eta_2+\xi_2\eta_1)+ 2z\xi_2\eta_2,$\\
	$z'=x\eta_1^2 +y\eta_1\eta_2+ z\eta_2^2,$\\
	of $GL(2,\mathbb{F}_q)$, consider its stabilizer group $St(C)$, the set $X_C$ of all algebras $A$ for which $\alpha_1^2+\alpha_2\beta_1=x_0, \alpha_1\alpha_2+\alpha_4\beta_1=y_0,\alpha_2^2-\alpha_1\alpha_4=z_0$, then count the number of non-isomorphic such algebras only with respect to the action of group $St(C)$. In this way we can count the number of non-isomorphic 5th set algebras. 
	Further we rely on the following known results:
	\begin{proposition} If $l, m$ are natural numbers, $G$ is a $m$-order cyclic group is acting on itself by $(x,g)\mapsto xg^l$, then the number of different orbits is equal to $gcd(l,m)$ and $G$ is the disjoint union of orbits $G^l, \sigma G^l,...,\sigma^{gcd(l,m)-1}G^l$, where $\sigma$ is a generator of $G$ and $G^l=\{g^l:\ g\in G \}$.\end{proposition}
	
	\begin{proposition} If $\nu\in \mathbb{F}_q$ is a non-square element, then the equation $x^2-y^2\nu =1$ has $q+1$ solutions.\end{proposition}
	\begin{proposition} If $\operatorname{char}(\mathbb{F})> 2$, then the following are the representatives and their stabilizer subgroups of second order symmetric matrices(quadratic forms) over $\mathbb{F}_q$ up to the congruence with respect to $GL(2, \mathbb{F}_q)$, where $\nu$ is a fixed non-square element,:
	\begin{enumerate}
		\item \textbf{Rank 0}
		\[
		S = \begin{pmatrix} 0 & 0 \\ 0 & 0 \end{pmatrix},\ \ St(S)=\text{GL}_2(\mathbb{F}_q).
		\]
		
		\item \textbf{Rank 1 (Square Discriminant)}
		\[
		S = \begin{pmatrix} 1 & 0 \\ 0 & 0 \end{pmatrix},\ \ St(S)=
		\left\{ \begin{pmatrix} \pm 1 & 0 \\ c & d \end{pmatrix} \;\middle|\; c \in \mathbb{F}_q,\, d \in \mathbb{F}_q^*\right\} 
		\]
		
		\item \textbf{Rank 1 (Non-Square Discriminant)}
		\[
		S = \begin{pmatrix} \nu & 0 \\ 0 & 0 \end{pmatrix},\ \ St(S)=
		\left\{ \begin{pmatrix} \pm1 & 0 \\ c & d \end{pmatrix} \;\middle|\; \, c \in \mathbb{F}_q,\, d \in \mathbb{F}_q^* \right\}.
		\]
		
		\item \textbf{Rank 2 (Hyperbolic / Split)}
		\[
		S = \begin{pmatrix} 0 & 1 \\ 1 & 0 \end{pmatrix},\ \ St(S)=
		\left\{ \begin{pmatrix} a & 0 \\ 0 & a^{-1} \end{pmatrix} \;\middle|\; a \in \mathbb{F}_q^* \right\} \cup \left\{ \begin{pmatrix} 0 & b \\ b^{-1} & 0 \end{pmatrix} \;\middle|\; b \in \mathbb{F}_q^* \right\}.
		\]
		
		\item \textbf{Rank 2 (Anisotropic -- Non-Square Determinant Class)}
		\[
		S = \begin{pmatrix} 1 & 0 \\ 0 & -\nu \end{pmatrix},\ -\nu\ is\ not\ a\ square, St(S)=O(2,q)=\{\begin{pmatrix} a & \pm c\nu \\ c & \pm a \end{pmatrix}: a^2-c^2\nu=1\}.\]
		
		\item \textbf{Rank 2 (Anisotropic -- Identity Representative)}
		\[
		S = \begin{pmatrix} 1 & 0 \\ 0 & 1 \end{pmatrix},\ \ St(S)=O_2(\mathbb{F}_q) = \{ g \in \text{GL}_2(\mathbb{F}_q) \mid g^T g = I \}.
		\] \end{enumerate}
\end{proposition}
Note that if $\operatorname{char}(\mathbb{F}_q)>2$, then \ref{sys:1} is equivalent to \[ \overline{A}'=g^T\overline{A}g,\ \mbox{where}\ \overline{A}=\begin{pmatrix} \alpha_1^2+\alpha_2\beta_1 & (\alpha_1\alpha_2+\alpha_4\beta_1)/2 \\ (\alpha_1\alpha_2+\alpha_4\beta_1)/2 & \alpha_2^2-\alpha_1\alpha_4 \end{pmatrix}.\]
Therefore, to count the number of non-isomorphic the fifth subset algebras in	
$\operatorname{char}(\mathbb{F}_q)>2$ case we can rely on Proposition 3.4. 
We deal with each of the above canonical representative $S$ separately.\
\begin{enumerate}
	\item \textbf{Rank 0}. 
	In this case the corresponding set of algebras is \[X_{0,1}=\{A: \alpha_1^2+\alpha_2\beta_1=\alpha_1\alpha_2+\alpha_4\beta_1=\alpha_2^2-\alpha_1\alpha_4=0\}=\]
	\[\{\begin{pmatrix}
		\alpha_1 &-\alpha_1^2/\beta_1 & -\alpha_1^2/\beta_1&\alpha_1^3/\beta_1^2 \\
		\beta_1 &-\alpha_1 &-\alpha_1 &\alpha_1^2/\beta_1
	\end{pmatrix}: \ \beta_1\neq 0\}\cup \{\begin{pmatrix}
		0 &0 &0&\alpha_4 \\
		0 &0 &0&0
	\end{pmatrix}: \alpha_4\in \mathbb{F}_q \}\]
	If $\beta_1\neq 0$ case due to \ref{sys:1} we have
	$$\beta' _1=\frac{1}{\Delta }\left(\beta _1 \xi _1^3-3 \alpha _1 \xi _1^2 \xi _2+3\alpha_1^2/\beta_1 \xi _1 \xi _2^2-\alpha_1^3/\beta_1^2\xi _2^3\right)=\frac{1}{\beta _1^2\Delta}(\beta _1 \xi _1-\alpha_1\xi_2)^3.$$ It shows that $\beta' _1$ can be made zero, that is one can assume that the original $\beta_1=0$.
	In this case $\alpha_2=\alpha_1=0$, so $A=
	\begin{pmatrix}
		0 &0 &0&\alpha_4 \\
		0 &0 &0 &0
	\end{pmatrix}$. If $\alpha_4=0$ we have the trivial algebra, if $\alpha_4\neq 0$ it can be made $1$ to have $A=\left(
	\begin{array}{cccc}
		0 &0 & 0&1 \\
		0 &0 &0 &0
	\end{array}
	\right)$.
	
	Conclusion: The total number of orbits in $X_{0.1}$ is $2$ and $\left(
	\begin{array}{cccc}
		0 &0 & 0&0 \\
		0 &0 &0 &0
	\end{array}
	\right), \left(
	\begin{array}{cccc}
		0 &0 & 0&1 \\
		0 &0 &0 &0
	\end{array}
	\right)$ can be taken as their representatives.
	
	\item \textbf{Rank 1 (Square Discriminant)}.
	In this case $X_{1,1}=$ \[\{A: \alpha_1^2+\alpha_2\beta_1=1, \alpha_1\alpha_2+\alpha_4\beta_1=\alpha_2^2-\alpha_1\alpha_4=0\}=\{A(\beta_1)=\begin{pmatrix}
		\pm1&0&0&0\\ \beta_1&\mp1&\mp1&0
	\end{pmatrix}: \ \beta_1\in \mathbb{F}_q\}.\]
	
	If $\operatorname{char}(\mathbb{F}_q)>3$, then  $gA(\beta_1)(g^{-1})^{\otimes 2}= \begin{pmatrix}1&0&0&0\\ 0&-1&-1&0\end{pmatrix}$, where $g^{-1}= \begin{pmatrix} \pm1 & 0 \\ \beta_1/3 & d \end{pmatrix}$,\\ if $\operatorname{char}(\mathbb{F}_q)=3$, then  $gA(\beta_1)(g^{-1})^{\otimes 2}= \begin{pmatrix}1&0&0&0\\ 1&-1&-1&0\end{pmatrix}$, for $g^{-1}= \begin{pmatrix} \pm1 & 0 \\ 1/\beta_1 & d \end{pmatrix}$.
	
	Conclusion: The total number of orbits in $X_{1.1}$ is $\begin{cases} 1,& if\ \operatorname{char}(\mathbb{F}_q)>3,\\ 2,& if\ \operatorname{char}(\mathbb{F}_q)=3.\end{cases}$ and $\begin{pmatrix}1&0&0&0\\ 0&-1&-1&0\end{pmatrix}$ ($\begin{pmatrix}1&0&0&0\\ 0&-1&-1&0\end{pmatrix}, \begin{pmatrix}1&0&0&0\\ 1&-1&-1&0\end{pmatrix}$) can be taken as their representatives in $1$ (respectively, $2$) orbit case, respectively.
	
	\item \textbf{Rank 1 (Non-Square Discriminant)}
	The corresponding set of algebras is \[X_{1,2}=\{A: \alpha_1^2+\alpha_2\beta_1=\nu, \alpha_1\alpha_2+\alpha_4\beta_1=\alpha_2^2-\alpha_1\alpha_4=0\}=\emptyset- \mbox{empty\ set}. \]  		
	It is due to $\beta_1\neq 0$, $\alpha_2=(\nu-\alpha_1^2)/\beta_1, \alpha_4=-\alpha_1(\nu-\alpha_1^2)/\beta_1^2$ and \\ $(\nu-\alpha_1^2)^2/\beta_1^2+\alpha_1^2(\nu-\alpha_1^2)/\beta_1^2=(\nu-\alpha_1^2)v/\beta_1^2=0$, which is a contradiction.\\
	Conclusion: The total number of orbits in $X_{1,2}$ is $0$.
	
	\item \textbf{Rank 2 (Hyperbolic / Split)}.
	The corresponding set of algebras is $X_{2,1}=$ \[\{A: \alpha_1^2+\alpha_2\beta_1=0, \alpha_1\alpha_2+\alpha_4\beta_1=2, \alpha_2^2-\alpha_1\alpha_4=0\}=\{A(\beta_1)=\begin{pmatrix}
		0 &0&0&2/\beta_1 \\
		\beta_1 &0 &0 &0
	\end{pmatrix}: \ \beta_1\neq 0\}.\]
	$$St(S)=H\cup Hg_0, \mbox{where}\ H=\{h(a)=\begin{pmatrix} a^{-1} & 0\\ 0 & a \end{pmatrix}: a\neq 0\}, g_0=\begin{pmatrix} 0 & 1 \\ 1 & 0 \end{pmatrix}.$$
	Application of $h^{-1}(a)$ ($g_0^{-1}$) to $A(\beta_1)$ shows that $A(\beta_1)\simeq A(a^3\beta_1)$ (respectively, $A(\beta_1)\simeq A(2/\beta_1)$). According to Proposition 3.2	if
	$q-1\equiv 1\pmod 3$ or $q-1\equiv 2\pmod 3$, then $\mathbb{F}^*_q=(\mathbb{F}^*_q)^3$, that is there is only one orbit and $A(1)$ can be taken as a representative. In $q\equiv 1\pmod 3$ case we have disjoint union $\mathbb{F}^*_q=(\mathbb{F}^*_q)^3\cup\sigma(\mathbb{F}^*_q)^3\cup \sigma^2(\mathbb{F}^*_q)^3$ and consider the corresponding algebras $A(1),A(\sigma), A(\sigma^2)$. So only one of the following may happen:\\ a) $2=b^3$. In this case $2/\sigma^2=(b/\sigma)^3\sigma$, that is $A(\sigma)\simeq A(\sigma^2)$ and we have only two different orbits, $A(1)$ and $A(\sigma)$ can be taken as their representatives.\\ b) $2=\sigma b^3$, $2/\sigma= b^3$, 
	that is $A(1)\simeq A(\sigma)$ and we have only two different orbits, $A(1)$ and $A(\sigma^2)$ can be taken as their representatives.\\ c) $2=\sigma^2 b^3$, that is $A(1)\simeq A(\sigma^2)$ and we have only two different orbits, $A(1)$ and $A(\sigma)$ can be taken as their representatives.

	Conclusion: The total number of orbits in $X_{2,1}$ is 
	$\begin{cases} 1,& if\ q\equiv 0\pmod 3\ or\ q\equiv 2\pmod 3,\\ 2,& if\ q\equiv 1\pmod 3.\end{cases}$.\\ 
	\item \textbf{Rank 2 (Anisotropic --- Non-Square Determinant Class)}. In this case 
	$$St(C)=O(2,\mathbb{F}_q)=H\cup Hg_0, \mbox{where}\ H=\{h(a,c)=\begin{pmatrix} a & c\nu \\ c & a \end{pmatrix}: a^2-c^2\nu=1\}, g_0=\begin{pmatrix} 1 & 0 \\ 0 & -1 \end{pmatrix},$$ \[X_{2,2}=\{A: \alpha_1^2+\alpha_2\beta_1=1, \alpha_1\alpha_2+\alpha_4\beta_1=0, \alpha_2^2-\alpha_1\alpha_4=-\nu\}=\] $\{A(\alpha_1,\beta_1)=\begin{pmatrix}\alpha_1&-\nu\beta_1&-\nu\beta_1&\nu\alpha_1\\ \beta_1&-\alpha_1&-\alpha_1&\nu\beta_1\end{pmatrix}: \alpha_1,\beta_1\in F_q, \alpha_1^2-\nu\beta_1^2=1 \}$ and 
	$q\equiv 1\pmod 4$ as far as only in this case both $\nu,-\nu$ are non-squares. 
	
	Due to \ref{sys:1} in $g^{-1}=h(a,c)$ case we have
	\begin{equation}	(\alpha'_1,\beta'_1)=(\alpha_1,\beta_1)\begin{pmatrix}
			a(4a^2-3)&- c\nu(4a^2-1) \\  -c\nu(4a^2-1)& a(4a^2-3)\end{pmatrix}.\end{equation}
	Note that $\vert X_{2,2}\vert =q+1$ due to Proposition 3.3.
	
	Consider the quadratic field extension $\mathbb{F}_{q^2} = \mathbb{F}_q[\sqrt{\nu}]$ and the norm map $N : \mathbb{F}_{q^2}^* \to \mathbb{F}_q^*$  defined by:
	\[
	N(x + y\sqrt{\nu}) = (x + y\sqrt{\nu})(x - y\sqrt{\nu}) = x^2 - y^2\nu.
	\]
	If we identify $A(x,y)$, $h(x,y)$, $(x,y)$ and $x + y\sqrt{\nu}$, then
	$X_{2,2} = \{x + y\sqrt{\nu}=(x,y)\in \mathbb{F}_q^2 : x^2 - y^2\nu = 1\}$ and 
	\[
	H = \{ w = x + y\sqrt{\nu} \in \mathbb{F}_{q^2}^* : N(w) = 1 \}
	\]
	
	For $z = a - c\sqrt{\nu} \in H$ we have:
	\[
	z^3 = a(4a^2 - 3) - c(4a^2 - 1)\sqrt{\nu}
	\]
	Moreover, we have a realization of action 3.1
	\[
	w \mapsto w \cdot z^3=(xa(4a^2 - 3) - yc(4a^2 - 1)\nu)+(ya(4a^2 - 3) - xc(4a^2 - 1))\sqrt{\nu}.
	\]
	as action of $H$, as a subgroup of $\mathbb{F}_{q^2}^*$, on itself by $w \mapsto w \cdot z^3$.
	
	According to Proposition 3.2 only in $q+1\equiv 0\pmod 3$, that is  $q\equiv 2\pmod 3$, case $H=H^3\cup\zeta H^3\cup\zeta^2H^3$, where $\zeta$ is a generator of cyclic group $H$, we have three $H$-orbits, otherwise there is only one orbit and $(1.0)$, that is
	$\begin{pmatrix}1&0&0&\nu\\ 0&-1&-1&0\end{pmatrix}$, can be taken as a representative. Due to $q\equiv 1\pmod 4$ the relation $q\equiv 2\pmod 3$ is equivalent to  $q\equiv 5\pmod {12}$  and $g_0$ doesn't change $H^3$ and swaps $\zeta H^3$ and $\zeta^2H^3$, so there are only two orbits. For their representatives $\begin{pmatrix}1&0&0&\nu\\ 0&-1&-1&0\end{pmatrix}$, $\begin{pmatrix}x_0&-\nu y_0&-\nu y_0&\nu x_0\\ y_0&-x_0&-x_0&\nu y_0\end{pmatrix}$ can be taken, where $x_0+y_0\sqrt{\nu}=\zeta$. Therefore in this case the total number of orbits is equal to $\begin{cases}1,& q\not\equiv 5\pmod {12}\ and\ q\equiv 1\pmod {4},\\ 2,& q\equiv 5\pmod {12}.\end{cases}$ 
	
	\item \textbf{Rank 2 (Anisotropic -- Identity Representative)}
	In this case $$ St(C)=O_2(\mathbb{F}_q)=H\cup Hg_0, \hbox{where}\ H=\{h(a,c)=\begin{pmatrix} a & - c \\ c & a \end{pmatrix}: a^2+c^2=1\}, g_0=\begin{pmatrix} 1 & 0 \\ 0 &-1 \end{pmatrix}$$ and $X_{2,3}=\{A: \alpha_1^2+\alpha_2\beta_1=1, \alpha_1\alpha_2+\alpha_4\beta_1=0, \alpha_2^2-\alpha_1\alpha_4=1\}=$ $$\{\begin{pmatrix}\alpha_1&\beta_1&\beta_1&-\alpha_1\\ \beta_1&-\alpha_1&-\alpha_1&-\beta_1\end{pmatrix}:
	\alpha_1^2+\beta_1^2=1\}.$$
\end{enumerate}
According to \ref{sys:1} in $g^{-1}=h(a,c)$ ($g^{-1}=g_0$) case we have
\begin{equation} (\alpha'_1,\beta'_1)=(\alpha_1,\beta_1)\begin{pmatrix} a(4a^2-3)& - c(4a^2-1) \\ c(4a^2-1)&  a(4a^2-3)\end{pmatrix} ((\alpha'_1,\beta'_1)=(\alpha_1,-\beta_1),\ \mbox{respectively} ).\end{equation}
Note that, if $q\equiv 1\pmod 4$ then $i^2=-1$ for some $i\in \mathbb{F}_q$ and the bijection map $(x,y)\mapsto u = x + iy \in \mathbb{F}_q^*$ (with inverse map $u\mapsto ((u+u^{-1})/2, (u-u^{-1})/(2i))$), from $X_{2,3}$ to $\mathbb{F}_q^*$, implies that $\vert X_{2,3}\vert =q-1$. 

If $q\equiv 3\pmod 4$, that is $-1$ is not a square, then consider $\mathbb{F}_{q^2} = \mathbb{F}_q[i]$ with $i^2 = -1$. In this case 
$\vert X_{2,3}\vert =q+1$ due to Proposition 3.3. Identifying $A(x,y)$, $h(x,y)$, $(x,y)$ and $x+iy$ we have
$(x,y) \mapsto w = x + iy \in H = \{w \in \mathbb{F}_{q^2}^* : w^{q+1} = 1\}$.

For $z = a - ic \in H$ we have $z^3 = a(4a^2-3) - c(4a^2-1)i$ and $(x+yi)z^3 = xa(4a^2-3) +yc(4a^2-1)+(-xc(4a^2-1)+ya(4a^2-3))i$. It is a realization of action 3.2 as action of $H$ on itself by
$w \mapsto wz^3$.

In $q\equiv 1\pmod 4$ case according to Proposition 3.2 in $H$ we have three $H$-orbits if and only if $q-1\equiv 0\pmod 3$, otherwise $H$ is one orbit, for a representative $\begin{pmatrix}1&0&0&-1\\ 0&-1&-1&0\end{pmatrix}$ can be taken. Note that $q\equiv 1\pmod 4$ and 
$q-1\equiv 0\pmod 3$ is equivalent to $q\equiv 1\pmod {12}$ and $H=H^3\cup \sigma H^3\cup\sigma^2H^3$. Moreover $g_0$ swaps $H$-orbits $\sigma H$ and $\sigma^2H$. For representatives of two orbits $\begin{pmatrix}1&0&0&-1\\ 0&-1&-1&0\end{pmatrix}$ and $\begin{pmatrix}x_0&y_0&y_0&-x_0\\ y_0&-x_0&-x_0&-y_0\end{pmatrix}$ can be taken, where 
$x_0=(\sigma+\sigma^{-1})/2, y_0=(\sigma-\sigma^{-1})/(2i)$. Therefore in   
$q\equiv 1\pmod 4$  case we have $\begin{cases}1,& if\ q\not\equiv 1\pmod {12},\\ 2,& if\ q\equiv 1\pmod {12} \end{cases}$ orbits.

Similarly, in $q\equiv 3\pmod 4$ case according to Proposition 3.2 in $H$ we have three $H$-orbits if and only if $q+1\equiv 0\pmod 3$, otherwise $H$ is one orbit, for a representative $\begin{pmatrix}1&0&0&-1\\ 0&-1&-1&0\end{pmatrix}$ can be taken. Relations $q\equiv 3\pmod 4$ and 
$q+1\equiv 0\pmod 3$ can be given as $q\equiv 11\pmod {12}$ and $H=H^3\cup \zeta H^3\cup \zeta^{2}H^3$, where 
$x_0+iy_0=\zeta$ is a generator of the cyclic group $H$.
But $g_0$ swaps orbits $\zeta H$ and $\zeta^{2}H$ and keeps $H$. Therefore in   
$q\equiv 3\pmod 4$  case we have $\begin{cases}1,& if\ q\not\equiv 11\pmod {12},\\ 2,& if\ q\equiv 11\pmod {12} \end{cases}$ orbits.
For representatives of these orbits $\begin{pmatrix}1&0&0&-1\\ 0&-1&-1&0\end{pmatrix}$ and $\begin{pmatrix}x_0&y_0&y_0&-x_0\\ y_0&-x_0&-x_0&-y_0\end{pmatrix}$ can be taken.

The following proposition is useful in counting the total number of algebras counted in (1)-(6) cases.
\begin{proposition}
	The following statements are true:\\
	$\begin{cases}q=p^k,\\ q\not\equiv 1\pmod {3},\\ q\equiv 1\pmod {4},\\q\not\equiv 5\pmod {12} \end{cases}$ is equivalent to  $p=3$. 
	$\begin{cases}q=p^k,\\ q\equiv 1\pmod {3},\\ q\equiv 5\pmod {12} \end{cases}$ is empty.\
	$\begin{cases}q=p^k,\\ q\not\equiv 1\pmod {3},\\ q\equiv 1\pmod {4} \end{cases}$ is equivalent to $ q\equiv 5\pmod {12} \ or\ p=3\ and\ k\ is\ even.$\\
	$\begin{cases}q=p^k,\\ q\equiv 1\pmod {3},\\ q\equiv 1\pmod {4} \end{cases}$ is equivalent to $q\equiv 1\pmod {12}.$\
	$\begin{cases}q=p^k,\\ q\not\equiv 1\pmod {3},\\ q\equiv 3\pmod {4} \end{cases}$\\ is equivalent to $q\equiv 11\pmod {12} \ or\ p=3\ and\ k\ is\ odd.$\\
	$\begin{cases}q=p^k,\\ q\equiv 1\pmod {3},\\ q\equiv 3\pmod {4} \end{cases}$ is equivalent to $q\equiv 7\pmod {12}.$\end{proposition}
The following is the number of the fifth subset non-isomorphic algebras in $\operatorname{char}(\mathbb{F}_q)>2$ case: \\
$\operatorname{char}(\mathbb{F}_q)=3$ case:\ $2+2+0+1+1+\begin{cases}1,& q=3^k,\ k\ is \ odd,\\ 0,&q=3^k,\ k\ is\ even \end{cases}=\begin{cases}7,& q=3^k,\ k\ is \ odd,\\ 6,&q=3^k,\ k\ is\ even \end{cases}$,\\
$\operatorname{char}(\mathbb{F}_q)>3$ case:\\ 
1. $q \not\equiv 1 \pmod{3}$ :\\
$2+1+0+1+0+ \begin{cases}1,& q\equiv 5\pmod {12},\\ 2,& q\equiv 11\pmod {12}\end{cases}=\begin{cases}5,& q\equiv 5\pmod {12},\\ 6,& q\equiv 11\pmod {12}\end{cases}.$\\ 
2. $q \equiv 1 \pmod{3}$:\\
$2+1+0+2+0+ \begin{cases}1,& q\equiv 1\pmod {12},\\2 ,& q\equiv 7\pmod {12}\end{cases}=\begin{cases}6,& q\equiv 1\pmod {12},\\7 ,& q\equiv 7\pmod {12}\end{cases}$.

So, in $\operatorname{char}(\mathbb{F}_q)>3$ case we have $\begin{cases}5,& q\equiv 5\pmod {12},\\ 6 ,& q\equiv 1\pmod {12}\ or\ q\equiv 11\pmod {12},\\ 7 ,& q\equiv 7\pmod {12}\end{cases}$

Now we consider $\text{char}(\mathbb{F}_q) = 2$ case.
In this case 3.1 becomes
\begin{equation}
	\begin{aligned} (\alpha'_1)^2+\alpha'_2\beta'_1=&(\alpha_1^2+\alpha_2\beta_1)\xi_1^2+(\alpha_1\alpha_2+\alpha_4\beta_1)\xi_1\xi_2+(\alpha_2^2+\alpha_1\alpha_4)\xi_2^2, \\ 
		\alpha'_1\alpha'_2+\alpha'_4\beta'_1=&(\alpha_1\alpha_2+\alpha_4\beta_1)(\xi_1\eta_2+\xi_2\eta_1), \\ 
		(\alpha'_2)^2+\alpha'_1\alpha'_4=&(\alpha_1^2+\alpha_2\beta_1)\eta_1^2+(\alpha_1\alpha_2+\alpha_4\beta_1)\eta_1\eta_2+(\alpha_2^2+\alpha_1\alpha_4)\eta_2^2. \\ 
	\end{aligned}
	\label{sys:2}
\end{equation}

If $\alpha_1\alpha_2+\alpha_4\beta_1=0$ and  $a^2=\alpha_1^2+\alpha_2\beta_1, a'^2=(\alpha'_1)^2+\alpha'_2\beta'_1, d^2=\alpha_2^2+\alpha_1\alpha_4, d'^2=(\alpha'_2)^2+\alpha'_1\alpha'_4$, then \ref{sys:2} becomes $(a',d')=(a,b)\begin{pmatrix} \xi_1 & \eta_1 \\ \xi_2 & \eta_2 \end{pmatrix}$. Therefore, we can make $(a',d')=(0,0)$ or $(a',d')=(1,0)$ with the stationary subgroup $GL(2,\mathbb{F}_q)$ and   $\{\begin{pmatrix} 1 & 0 \\ c & d \end{pmatrix}: \ d\neq 0\}$, respectively. The corresponding set of algebras is \[ X_{0,0,0}=\{A: \alpha_1^2+\alpha_2\beta_1=\alpha_1\alpha_2+\alpha_4\beta_1=\alpha_2^2-\alpha_1\alpha_4=0\}=\]
\[\{\begin{pmatrix}
	\alpha_1 &-\alpha_1^2/\beta_1 & -\alpha_1^2/\beta_1&\alpha_1^3/\beta_1^2 \\
	\beta_1 &-\alpha_1 &-\alpha_1 &\alpha_1^2/\beta_1
\end{pmatrix}: \ \beta_1\neq 0\}\cup \{\begin{pmatrix}
	0 &0 &0&\alpha_4 \\
	0 &0 &0&0
\end{pmatrix}: \alpha_4\in \mathbb{F}_q \}\] and $X_{1,0,0}= \{A: \alpha_1^2+\alpha_2\beta_1=1, \alpha_1\alpha_2+\alpha_4\beta_1=\alpha_2^2-\alpha_1\alpha_4=0\}=$ \[\{\begin{pmatrix}
	1&0&0&0\\ \beta_1&1&1&0
\end{pmatrix}: \ \beta_1\in \mathbb{F}_q\},\ \mbox{respectively}.\] It is not difficult to see that in $X_{0,0,0}$ we have two orbits represented by $\begin{pmatrix}
	0 &0 &0&0 \\
	0 &0 &0&0
\end{pmatrix}$ and $\begin{pmatrix}
	0 &0 &0&1 \\
	0 &0 &0&0
\end{pmatrix}$, in $X_{1,0,0}$ we have one orbit represented by $\begin{pmatrix}
	1 &0 &0&0 \\
	0 &1 &1&0
\end{pmatrix}$. 

If $\alpha_1\alpha_2+\alpha_4\beta_1\neq 0$ we can make $\alpha'_1\alpha'_2+\alpha'_4\beta'_1=1$, then consider the following:\\ 
$(\alpha'_1)^2+\alpha'_2\beta'_1=(\alpha_1^2+\alpha_2\beta_1)\xi_1^2+\xi_1\xi_2+(\alpha_2^2+\alpha_1\alpha_4)\xi_2^2=(a\xi_1+d\xi_2)^2+\xi_1\xi_2$, \\ 
$\alpha'_1\alpha'_2+\alpha'_4\beta'_1=\alpha_1\alpha_2+\alpha_4\beta_1=1$, \\ 
$(\alpha'_2)^2+\alpha'_1\alpha'_4=(\alpha_1^2+\alpha_2\beta_1)\eta_1^2+\eta_1\eta_2+(\alpha_2^2+\alpha_1\alpha_4)\eta_2^2=(a\eta_1+d\eta_2)^2+\eta_1\eta_2$.

If $a=d=0$ we can make $(\alpha'_2)^2+\alpha'_1\alpha'_4=\xi_1\xi_2=\eta_2\eta_1=(\alpha'_2)^2+\alpha'_1\alpha'_4= 0$ with stationary subgroup $G=H\cup Hg_0$, where  $H=\{
h(a)=\begin{pmatrix} a & 0 \\ 0 & a^{-1} \end{pmatrix} :\ a \in \mathbb{F}_q^*\},$ $g_0=\begin{pmatrix} 0 & 1 \\ 1 & 0 \end{pmatrix}$, acting on \[X_{0,1,0}=\{A: \alpha_1^2+\alpha_2\beta_1=0, \alpha_1\alpha_2+\alpha_4\beta_1=1, \alpha_2^2-\alpha_1\alpha_4=0\}=\]
\[\{A(\beta_1)=\begin{pmatrix}
	0 &0 &0 & 1/\beta_1 \\
	\beta_1 & 0 &0 &0
\end{pmatrix}: \ \beta_1\neq 0\}.\]
$g^{-1}=h(a)$ case provides $A(\beta_1)\simeq A(a^3\beta_1)$ and $g^{-1}=g_0$ case provides $A(\beta_1)\simeq A(1/\beta_1)$. Therefore, due to Proposition 3.2, in $k$ odd case there is only one orbit, for a representative $A(1)$ can be taken. 

In $k$ even case we have only three $H$-orbits presented by $A(1), A(\sigma), A(\sigma^2)$ and $g_0$ doesn't change the first $H$-orbit while swaps the others: $A(\sigma^2)\simeq A(1/\sigma^2)=A((1/\sigma)^3\sigma)$. So, there are only two $G$-orbits and for the representatives $A(1), A(\sigma)$ can be taken.

If $a\neq 0, d=0$, then\\
$a'^2=\xi_1(a^2\xi_1+\xi_2)$, \\ 
$d'^2=\eta_1(a^2\eta_1+\eta_2)$, where $\xi_1\eta_2+\xi_2\eta_1=1$, \\  
and $\xi_1=a^{-1}, \xi_2=0, \eta_1=0, \eta_2=a$ makes $(a',d')=(1,0)$. Similarly, in $d\neq 0, a=0$ case also we can made $(a',d')=(1,0)$. The corresponding stationary subgroup consists of $h(a)=\begin{pmatrix} a & 0 \\ (1-a^2)/a & a^{-1} \end{pmatrix}$, where $a\neq 0$, and $\begin{pmatrix} a & a \\ (1-a^2)/a & a\end{pmatrix}=h(a)g_0$, where $g_0=\begin{pmatrix} 1 & 1 \\ 0 & 1\end{pmatrix}$, acting on
\[X_{1,1,0}=\{A: \alpha_1^2+\alpha_2\beta_1=1, \alpha_1\alpha_2+\alpha_4\beta_1=1, \alpha_2^2-\alpha_1\alpha_4=0\}=\]
\[\{A(\alpha_1)=\begin{pmatrix}
	\alpha_1 & \alpha_1 &\alpha_1 & \alpha_1 \\
	(1-\alpha_1^2)/\alpha_1 & \alpha_1 & \alpha_1 & \alpha_1
\end{pmatrix}: \ \alpha_1\neq 0\}.\] Consideration of $h(a)$ and $g_0$ provides that 
$ A(\alpha_1)\simeq A(\alpha_1a^{-3})$, $A(\alpha_1)\simeq A(1/\alpha_1)$. In a similar way as before we get: If $k$ is odd, then $X_{1,1,0}$ is one orbit presented by $A(1)=\begin{pmatrix}
	1 & 1 &1 &1 \\
	0&1 &1 &1
\end{pmatrix}$,  if $k$ is even, then $X_{1,1,0}$ consists of  two orbits presented by $A(1)$ and $A(\sigma)$.

If $ad\neq 0,$ then  at $\xi_1=a^{-1}, \xi_2=0, \eta_1=0, \eta_2=a$, then $a'=1,d'=a/d$.
Therefore, it is enough to consider  $(1,d)$ case, where $d\neq 0.$ Note that if:\\
$\xi_1^2+d^2\xi_2^2+\xi_1\xi_2=1$, \\ 
$\xi_1\eta_2+\xi_2\eta_1=1$, \\ 
$\eta_1^2+\eta_1\eta_2+d^2\eta_2^2=d'^2$, \\
then $Tr(d)=Tr(d')$, where $Tr(x)=\sum_{i=0}^{k-1}x^{2^i}$ .  
Indeed, if $\xi_2=0$, then the third equality implies $Tr(d)=Tr(d')$. If $\xi_2\neq 0$, then the first equation implies $(\xi_1\xi_2)^2+\xi_1\xi_2+d^2=(1/\xi_2)^2$, in particular, $Tr(d)=Tr(1/\xi_2)$. By the use of the second equality $\eta_1=(1+\xi_1\eta_2)/\xi_2$ the third equality can be represented as $\eta_1^2+\eta_1\eta_2+d^2\eta_2^2=
(1+\xi_2^2\eta_2^2)/\xi_2^2+\eta_2(1+\xi_2\eta_2)/\xi_2+d^2\eta_2^2=(1+\xi_1^2\eta_2^2+\xi_2\eta_2+\eta_2^2(\xi_1\xi_2+d^2\xi_2^2))/\xi_2^2=(1+\xi_1^2\eta_2^2+\xi_2\eta_2+\eta_2^2(1+\xi_1^2)/\xi_2^2
=(1/\xi_2)^2+\eta_2/\xi_2+(\eta_2/\xi_2)^2=d'^2$ (*). Therefore, $Tr(d')=Tr(1/\xi_2)=Tr(d)$. Note that $Tr$ is a $\mathbb{F}_2$-linear map of range $\{0,1\}$, has property  $Tr(x^2)=Tr(x)$ and $x^2+x=d$ has solution if and only if $Tr(d)=0$.

It implies that if $Tr(d)=0$, then one can make $d'=0$. In this case we come to the before considered $X_{1,1,0}$. If $Tr(d)=1$, then we can make $d'=\epsilon$, where $\epsilon\in F_q$ is a fixed element with $Tr(\epsilon)=1$. Indeed, $\xi_1=1,\xi_2^2=0,\eta_2=1, \eta_1$, where $\eta_1^2+\eta_1+d^2=\epsilon^2$, makes $d'=\epsilon$.
Let us solve\\
$\xi_1^2+\epsilon^2\xi_2^2+\xi_1\xi_2=1$, \\ 
$\xi_1\eta_2+\xi_2\eta_1=1$, \\ 
$\eta_1^2+\eta_1\eta_2+\epsilon^2\eta_2^2=\epsilon^2$\\ to find the corresponding stabiliser subgroup.
If $\xi_2\neq 0$, then (*) and the first equalities imply that $1+\eta_2\xi_2+\eta_2^2+\epsilon^2\xi_2^2=0$, $1+\eta_2\xi_2+\eta_2^2+1+\xi_1^2+\xi_1\xi_2=0$,  $(\eta_2+\xi_1)(\eta_2+\xi_1+\xi_2)=0$. In $\eta_2+\xi_1=0$ case we have $\begin{pmatrix}\xi_1&\xi_1+\epsilon^2\xi_2\\ \xi_2 & \xi_1 \end{pmatrix}$.
If $\eta_2=\xi_1+\xi_2$, then we have $\begin{pmatrix}\xi_1& \epsilon^2\xi_2\\ \xi_2 & \xi_1+\xi_2 \end{pmatrix}$.

In other words $G=H\cup Hg_0$, where $H=\{h(a,b)=\begin{pmatrix}a& \epsilon^2b\\ b & a+b \end{pmatrix}:\ a^2+\epsilon^2b^2+ab=1\}$,$g_0=\begin{pmatrix}1& 1\\ 0 & 1 \end{pmatrix}$.

This group $G$ is acting on  
\[X_{1,1,\epsilon}=\{A: \alpha_1^2+\alpha_2\beta_1=1, \alpha_1\alpha_2+\alpha_4\beta_1=1, \alpha_2^2-\alpha_1\alpha_4=\epsilon^2\}=\]
\[\{A(\alpha_1,\beta_1)=\begin{pmatrix}
	\alpha_1 & \alpha_1+\epsilon^2\beta_1 &\alpha_1+\epsilon^2\beta_1 & \alpha_1+\epsilon^2(\beta_1+\alpha_1) \\
	\beta_1 & \alpha_1 & \alpha_1 & \alpha_1+\epsilon^2\beta_1\end{pmatrix}: \ \alpha_1\beta_1+\alpha_1^2+\epsilon^2\beta_1^2=1\}.\] In particular,  $\vert X_{1,1,\epsilon}\vert=\vert H\vert$. In $g^{-1}=h(a,b)$ case we have \\
$\alpha'_1=\alpha_1(b^3+a(1+b^2))+\beta_1b\epsilon^2(1+b^2),\		\beta'_1=\alpha_1b(1+b^2)+\beta_1(b+a(1+b^2))$, that is
$$\begin{pmatrix}\alpha'_1\\ \beta'_1\end{pmatrix}=\begin{pmatrix}b^3+a(1+b^2)&b\epsilon^2(1+b^2)\\ b(1+b^2)&b+a(1+b^2)\end{pmatrix}\begin{pmatrix}\alpha_1\\ \beta_1\end{pmatrix}.$$.

Note that the determinant of $\begin{pmatrix}b^3+a(1+b^2)+1&b\epsilon^2(1+b^2)\\ b(1+b^2)&b+a(1+b^2)+1\end{pmatrix}$ is $b(1+b^2)$ and it is zero only if $b=0$ or $b=1$. Therefore $H$-stabilizer of $(x,y)\in X_{1,1,\epsilon}$ is trivial whenever $t^2+t+\epsilon^2=1$ has no root, which happens if and only if $k$ is odd. In this case $X_{1,1,\epsilon}$ is one orbit, for a representative one can take $A(1,0)$. If $\lambda$ is a root for $t^2+t+d^2=1$, that is $k$ is even, then the stabilizer has two more elements: $h(\lambda,1)$ and $h(\lambda+1,1)$. But $g_0$ swaps $(\lambda,1)$ and $(\lambda+1,1)$. Therefore in this case there two $G$-orbits, for representatives of which $A(1,0), A(\lambda,1)$ can be taken.

So, the total number of the fifth subset non-isomorphic algebras in $\operatorname{char}(\mathbb{F}_q)=2$ case is 
$$\begin{cases}6,& k\ is\ odd\\ 8,& k\ is\ even\end{cases}.$$

The obtained result can be summarised in the following form. 
\begin{theorem} The number all non-isomorphic the fifth subset  algebras over a finite field $\mathbb{F}_q$, where $q=p^k$, is:
	\begin{itemize}	
		\item $\operatorname{char}(\mathbb{F}_q)>3$ case:	
		$\begin{cases}5,& q\equiv 5\pmod {12},\\ 6 ,& q\equiv 1\pmod {12}\ or\ q\equiv 11\pmod {12},\\ 7 ,& q\equiv 7\pmod {12}\end{cases}$
		\item $\operatorname{char}(\mathbb{F}_q)= 3$ case:
		$\begin{cases}6,&if\ k\ is\ even\\  7,&if\ k\ is\ odd\end{cases}.$	
		\item $\operatorname{char}(\mathbb{F}_q)= 2$ case:	
		$\begin{cases}6,&if\ k\ is\ odd\\ 8,&if\ k\ is\ even\end{cases}.$
	\end{itemize}
\end{theorem}

Our final result on the number of non-isomorphic two-dimensional algebras over a finite field $F_q$ is the following Theorem.

\begin{theorem} The number all non-isomorphic two-dimensional algebras over a finite field $\mathbb{F}_q$, where $q=p^k$, is:
	\begin{itemize}	
		\item $\operatorname{char}(\mathbb{F}_q)>3$ case:	
		$q^4+q^3+4q^2+4q+\begin{cases}6,& q\equiv 5\pmod {12},\\ 7 ,& q\equiv 1\pmod {12}\ or\ q\equiv 11\pmod {12},\\ 8 ,& q\equiv 7\pmod {12}\end{cases}.$
		\item $\operatorname{char}(\mathbb{F}_q)= 3$ case:
		$q^4+q^3+4q^2+4q+\begin{cases}6,&if\ k\ is\ even\\  7,&if\ k\ is\ odd\end{cases}.$	
		\item $\operatorname{char}(\mathbb{F}_q)= 2$ case:	
		$q^4+q^3+4q^2+3q+\begin{cases}6,&if\ k\ is\ odd\\ 8,&if\ k\ is\ even\end{cases}.$
	\end{itemize}
\end{theorem}
\begin{remark} In the fifth subset case, the canonical representatives of the non-isomorphic classes of algebras are presented in the proof of Theorem 3.1. In the finite field case, such canonical representatives for the first four classes of algebras can be derived directly from Theorem 2.1. \end{remark}

\end{document}